\documentclass[12pt]{amsart}   
\usepackage{listings}
\usepackage{amsmath}
\usepackage{amssymb,amsfonts}   
\usepackage{graphicx}	        

\newtheorem{theorem}{Theorem}[section]
\newtheorem{proposition}[theorem]{Proposition}
\newtheorem{lemma}[theorem]{Lemma}
\newtheorem{definition}[theorem]{Definition}

\newtheorem{corollary}[theorem]{Corollary}

\usepackage[top=1.00in, bottom=.8in, left=.9in, right=.9in]{geometry} 
\usepackage{amsmath,graphics,amsfonts,amsthm,amssymb,color}
\usepackage{graphics}
\usepackage{latexsym}
\usepackage{multicol}

\title{Almost Sure Local Well-Posedness for the Supercritical Quintic NLS}
\author{
        Justin T. Brereton \\
}
\date{\today
}
\begin{document}

\begin{abstract}
This paper studies the quintic nonlinear Schr\"odinger equation on $\mathbb{R}^d$ with randomized initial data below the critical regularity $H^{\frac{d-1}{2}}$. The main result is a proof of almost sure local well-posedness given a Wiener Randomization of the data in $H^s$ for $s \in (\frac{d-2}{2}, \frac{d-1}{2})$. The argument further develops the techniques introduced in the work of \'A. B\'enyi, T. Oh and O. Pocovnicu on the cubic problem. The paper concludes with a condition for almost sure global well-posedness.
\end{abstract}

\maketitle

\section{Introduction} \label{mathrefs}
Consider the Cauchy problem for the nonlinear Schr\"odinger equation. Given initial data $\phi \in H^s(\mathbb{R}^d)$, for $(t,x) \in \mathbb{R} \times \mathbb{R}^d$ the solution $u(t,x) \in \mathbb{C}$ satisfies  
\begin{equation} \label{eq:1}
\begin{split}
iu_t + \Delta u &= \pm |u|^{p-1}u \\
u|_{t=0} &= \phi \\
\end{split}
\end{equation} where $+$ and $-$ correspond to the defocusing and focusing cases respectively. 
This equation has conserved mass and energy 
\begin{equation*}
\begin{split}
M(t) &= \frac{1}{2} \int_{\mathbb{R^{d}}} |u(t,x)|^2 dx\\
E(t) &= \frac{1}{2}  \int_{\mathbb{R^{d}}} |\nabla u(t,x)|^2 dx \pm \frac{1}{p+1}  \int_{\mathbb{R^{d}}} |u(t,x)|^{p+1} dx .\\ 
\end{split}
\end{equation*} 

The NLS equation is also invariant under a dilation symmetry. Given $u(t,x)$ that solves \eqref{eq:1}, $u_{\lambda}(t,x) = \lambda^{2/(p-1)}u(\lambda^2t,\lambda x)$ is a solution for every $\lambda$. Furthermore there is a Sobolev index $s_c = \frac{d}{2} - \frac{2}{p-1}$ such that the homogoneous Sobolev norm $\|u_{\lambda}\|_{\dot{H}^{s_c}}$ is constant under this scaling. This index $s_c$ is known as the scaling critical index, and when $\frac{d}{2} - \frac{2}{p-1} = s_c=1$ the problem is known as energy critical, since the energy scales like $\dot{H}^{s_c} = \dot{H}^1$. Given initial data $\phi \in H^s(\mathbb{R}^d)$, the problem is called subcritical when $s > s_c$ and supercritical when $s < s_c$. 

In addition, special pairs of exponents $(q,r)$ satisfying the bounds $2 \le q,r \le \infty$ and \newline $(q,r,d) \ne (2,\infty,2)$ are called Schr\"odinger-admissible if 
\begin{equation} \label{eq:3}
\frac{2}{q} + \frac{d}{r} = \frac{d}{2} .\\
\end{equation}
For such a pair we have the well known Strichartz estimate 
\begin{equation} \label{eq:5}
\|S(t)\phi\|_{L^q_t L^r_x(\mathbb{R}\times \mathbb {R}^d)} \le C\|\phi\|_{L^2(\mathbb{R^d})} \\
\end{equation}
where $S(t)$ denotes the linear Schr\"odinger semigroup operator $e^{it \Delta}$ that corresponds to solving the linear Schr\"odinger equation for time $t$, see \cite{Strichartz}, \cite{Yajima}.

It is known that the NLS equation is ill-posed in the supercritical case; for such $s$ one can construct special initial data $\phi \in H^s(\mathbb{R}^d)$ such that for every $T>0$, \eqref{eq:1} has no solution on $(-T,T)$ that stays in $H^s(\mathbb{R}^d)$, as demonstrated in \cite{Alazard}.  Though local well-posedness is not guaranteed, it is important to determine if there are solutions for most supercritical intial data $\phi$. This leads one to investigate the problem of almost sure well-posedness for initial data chosen for supercritical randomized initial data. Pocovnicu, B\'enyi, and Oh have proven almost sure local well-posedness for the energy critical $\mathbb{R}^4$ problem using $X^{s,b}$ spaces in \cite{Pocovnicu}. They then proved a separate result for the cubic equation for all $d \ge 3$ using $U^p$ and $V^p$ spaces and their adaptations for the Schr\"odinger equation in \cite{Pocovnicu2}. 
  
In this paper we adapt the techniques of \cite{Pocovnicu} and \cite{Pocovnicu2} in order to prove local well-posedness in the quintic case for dimension $d \ge 3$. Following \cite{Pocovnicu} we apply a Wiener Randomization to the initial data $\phi \in H^s(\mathbb{R}^d)$. This randomization method takes a function $\phi \in H^s(\mathbb{R}^d)$ and for each $\omega$ in a probability space $\Omega$ produces a randomized function 
\begin{equation} \label{eq:6}
\phi^{\omega} = \sum_{n \in \mathbb{Z}^d} g_n(\omega)\eta(D-n)\phi
\end{equation}
that is in $H^s(\mathbb{R}^d)$ with probability $1$ but gains regularity with probability $0$. The $g_n(\omega)$ are mean zero, i.i.d. complex random variables that are required to satisfy a decay condition, the Gaussian being such a random variable. The term $\eta(D-n)$ is a Fourier multiplier whose symbol approximates the characteristic function of the unit cube centered at $n$ in frequency space. 

In section 2 we present several previously known probabilistic bounds on the Wiener randomization $\phi^{\omega}$ of $\phi \in H^s(\mathbb{R}^d$ as well as its linear Schr\"odinger evolution $S(t)\phi^{\omega}$. One of these is a probabilistic bound on $\| \langle \nabla \rangle^{s} S(t)\phi^{\omega}\|_{L^qL^r(I \times \mathbb{R}^d)}$ for arbitrarily large values of $q,r$. For large enough values of $q,r$ this is a norm that scales subcritically, which means we can approach almost sure local well-posedness as if it is a subcritical problem. 

Our main result is the almost sure local well-posedness of \eqref{eq:1} with initial data $\phi^{\omega}$ chosen via the Wiener randomization of any $\phi \in H^s(\mathbb{R}^d)$:
\begin{theorem} \label{thm1}
Fix a dimension $d \ge 3$ and $s \in (\frac{d-2}{2}, \frac{d-1}{2})$. Given $\phi \in H^s(\mathbb{R}^d)$ with Wiener randomization $\phi^{\omega}$, $\omega \in \Omega$, the quintic nonlinear Schr\"odinger equation is almost surely locally well-posed. More specifically, there exist $c_1, c_2, \theta>0$, such that for sufficiently small $T \ll1$, there is a set $\Omega_T \subset \Omega$ such that $P(\Omega_T) \ge 1- c_1e^{-c_2/T^{\theta}\|\phi\|_{H^{s}}}$ and for each $\omega \in \Omega_T$, the initial value problem 
\begin{equation*}
\begin{split}
iu_t+ \Delta u &= \pm |u|^4u\\
u(0) &= \phi^{\omega} \\
\end{split}
\end{equation*}
has a unique solution in the function class $C((-T,T) \rightarrow H^s(\mathbb{R}^d))$.
\end{theorem}

We now provide a brief outline of the proof. In section 3 we define the Littlewood-Paley projection operator, as well as the $U^2$ and $V^2$ spaces and their Schr\"odinger analogues, and in section 4 we present Strichartz estimates as well as a bilinear estimate for these spaces. The next step is to split the NLS solution $u$ into it's linear part $z(t) = S(t)\phi^{\omega}$ and nonlinear part 
\begin{equation}
v(t)= \pm \int_{0}^t -iS(t-t')[|v+z|^4(v+z)](t')dt' ,
\end{equation} the integral term of Duhamel's formula. Our probabilistic bounds tells us that $z$ almost surely has the same regularity as the initial data $\phi^{\omega}$. Therefore the linear part of the solution is almost surely in the supercritical space $H^s(\mathbb{R}^d)$, and it remains to prove existence of the nonlinear part $v(t)$. As mentioned earlier, $z(t)$ is bounded in subcritical norms, which means we can treat our linear solution $z(t)$ as a subcritical perturbative term in the Cauchy problem 
\begin{equation} \label{non}
\begin{split}
iv_t + \Delta v &= \pm (v+z)|v+z|^4 \\
v(0)  &= 0 \\
\end{split}
\end{equation}
that is satisfied by the nonlinear part $v$.

This means almost sure local well-posedness of $v(t)$ is essentially a subcritical problem. We prove local existence of the nonlinear part $v(t)$ using a fixed point argument based on doing a frequency decompostion of $v(t)$ and bounding it at each frequency.

Global well-posedness is a much harder problem. There is yet to be a proof of almost sure global well-posedness of any supercritical NLS problem. 
Pocovnicu, B\'enyi, and Oh proved almost sure global well-posedness of $v \in H^1(\mathbb{R}^4)$ for the cubic problem under the assumption that there is a probabilistic bound on $\|v\|_{L^{\infty}H^1(\mathbb{R} \times\mathbb{R}^4)}$ in \cite{Pocovnicu2}. It seems difficult to prove such a bound. 

One could probably prove a similar result for the $3$ dimensional quintic problem, the energy critical dimension for the quintic problem. Instead we prove almost sure global well-posedness of $v$ in the subcritical space ${S}^{1+c}(\mathbb{R} \times \mathbb{R}^3)$ assuming the norm $\|v\|_{L^{10}L^{10}([-T,T]\times \mathbb{R}^3)}$ does not blow up in finite time. This means that a probabilistic a priori estimate for $\|v\|_{L^{10}L^{10}([-T,T]\times \mathbb{R}^3)}$ implies almost sure global well-posedness as expressed in the following result:

\begin{theorem} \label{thm2}
Assume $\frac{7}{8} < s < 1$ and $0 < c < \frac{1}{8}$. Suppose we have a probabilistic a priori estimate for $\|v\|_{L^{10}L^{10}([-T,T]\times \mathbb{R}^3)}$, meaning for every $T,R > 0$ there is a function $\alpha(T,R)$ and a set $\Omega'_{T,R}$ such that 
\begin{itemize}

\item{} For any $\omega \in \Omega'_{T,R}$, if the solution $v(t)$ to (\ref{non}) exists on $(-T,T)$ then we have the bound 
\begin{equation*}
\|v\|_{L^{10}L^{10}([-T,T]\times \mathbb{R}^3)}<R
\end{equation*}

\item{} $P(\Omega_{T,R}') \ge 1 -\alpha(T,R)$

\item{} $\forall T>0: \lim_{R\rightarrow \infty} \alpha(T,R) = 0$.

\end{itemize}
Then given $\phi \in H^s(\mathbb{R}^3)$ with Wiener randomization $\phi^{\omega}$, the initial value problem 
\begin{equation*}
\begin{split}
iu_t+ \Delta u &= \pm |u|^4u\\
u(0) &= \phi^{\omega} \\
\end{split}
\end{equation*}
is almost surely globally wellposed, meaning there is a set $\Omega_{T,R} \subset \Omega$ and constants $c_1,c_2,c_3 >0$ such that 
\begin{equation*}
P \left( \Omega_{T,R} \right) \ge 1-c_1e^{-c_2R^2} - c_3\alpha(T,R)
\end{equation*}
and for any $\omega \in \Omega_{T,R}$ the above equation has a unique solution in the function class \newline $C((-T,T) \rightarrow H^s(\mathbb{R}^d))$ with $v(t) \in H^{1+c}(\mathbb{R}^3)$ for any time $t \in (-T,T)$.
\end{theorem}

\section{Randomization of Initial Data and Probabilistic Estimates} \label{mathrefs}

Our method of randomization is the Wiener decomposition of the frequency space that was used in \cite{Pocovnicu}. Consider a Schwartz class function $\psi \in \mathcal{S}(\mathbb{R}^d)$ that approximates the cube of unit length centered at the origin in $\mathbb{R}^d$, meaning that $\psi$ is supported on $[-1,1]^d$ and \newline $\displaystyle \sum_{n \in \mathbb{Z}^d} \psi(\xi -n)$ is identically $1$. Then for each $n$, define the fourier multiplier $\eta$ as
\begin{equation} \label{eq:11}
\eta(D-n)u(x) = \mathcal{F}^{-1}[\psi(\xi-n)\mathcal{F}u] .
\end{equation}  
Note that this satisfies $\displaystyle \sum_{n \in \mathbb{Z}^d} \eta(D-n)u(x) = u(x)$. This provides a decomposition of the function $u$ into pieces whose frequencies are localized to cubes. 

The idea is then to consider a function $\phi \in H^s(\mathbb{R}^d)$ and for each $\omega$ from a probability space $\Omega$ create a randomized function $\displaystyle \sum_{n \in \mathbb{Z}^d} g_n(\omega)\eta(D-n)\phi$ for some random variables $g_n$. For each $n \in \mathbb{Z}^d$ let $\mu_n$ and $\nu_n$ be probability distributions on $\mathbb{R}$, symmetric about $0$, such that for some constant $c$ we have 
\begin{equation} \label{eq:14}
\begin{split}
\left| \int_{\mathbb{R}^d} e^{\lambda x}d\mu_n(x) \right| &\le e^{c\lambda^2} \\
\left| \int_{\mathbb{R}^d} e^{\lambda x}d\nu_n(x) \right| &\le e^{c\lambda^2} \\
\end{split}
\end{equation}
for all $n \in \mathbb{Z}^d$, $\lambda \in \mathbb{R}$. A Gaussian random variable would be an example of a random variable with these properties. Then define each $g_n$ to be an independent, mean zero, complex random variable on $\Omega$ such that $\text{Re}(g_n)$ and $\text{Im}(g_n)$ have distributions $\mu_n, \nu_n$. We define the Wiener randomization $\phi^{\omega}$ of $\phi \in H^s(\mathbb{R}^d)$ to be 
\begin{equation} \label{eq:15}
\phi^{\omega} = \sum_{n \in \mathbb{Z}^d} g_n(\omega)\eta(D-n)\phi .\\ 
\end{equation}

The main advantage derived from the Wiener Randomization is improved $L^p(\mathbb{R}^d)$ estimates on the randomized initial data $\phi^{\omega}$ off a small set, as a result of a stronger Bernstein's inequality. Despite only requiring that $\phi$ be in $H^s$, the randomized $\phi^{\omega}$ is in $L^P (\mathbb{R}^d)$ with probability $1$. In addition we have a probabilistic bound on $\|\phi^{\omega}\|_{H^s(\mathbb{R}^d)}$, which implies that $\phi^{\omega} \in H^s(\mathbb{R}^d)$ almost surely. 

We have the following key bounds on $\phi^{\omega}$ and its linear Schr\"odinger evolution with proofs from \cite{Pocovnicu}. I omit the proof of the second and third.  
For all $R>0, s>0$, and $\phi \in H^s(\mathbb{R}^d)$ we have:
\begin{equation*} 
\begin{split}
 P\left ( \|\phi^{\omega}\|_{H^s(\mathbb{R}^d)} > R \right ) &\le c_1e^{-c_2R^2/\|\phi\|^2_{H^s(\mathbb{R}^d)}}, \\
 P\left ( \|S(t)\phi^{\omega}\|_{L_t^qL_x^r([0,T]\times \mathbb{R}^d}) > R \right ) &\le c_1e^{-c_2R^2/T^{2/q}\|\phi\|^2_{L^2(\mathbb{R}^d)}}, \\
 P\left (\|\phi^{\omega}\|_{L^p(\mathbb{R}^d)} > R \right ) &\le c_1e^{-c_2R^2/\|\phi\|_{L^2(\mathbb{R}^d)}} .\\
\end{split}
\end{equation*}
\begin{lemma} \label{thm2.1} 
Given $\phi \in H^s$ with randomization $\phi^{\omega}$, for all $R>0$ there exist positive constants $c_1, c_2$ such that: 
\begin{equation}\label{eq:18}
 P\left ( \|\phi^{\omega}\|_{H^s(\mathbb{R}^d)} > R \right) \le c_1e^{-c_2R^2/\|\phi\|^2_{H^s(\mathbb{R}^d)}}. \\
\end{equation} 
\end{lemma}
\begin{proof} The proof is taken from \cite{Pocovnicu}. By Minkowski's Inequality, we have for $p \ge 2$,
\begin{equation}
\begin{split} \label{21}
\mathbb{E}[\|\phi^{\omega}\|^p_{H^s(\mathbb{R}^d)} ] &\le \left ( \left \| \|\langle \nabla \rangle^s \phi^{\omega} \|_{L^p(\Omega)} \right \|_{L^2(\mathbb{R}^d)} \right )^p \\
&= \left ( \left \| \| \sum_{n \in \mathbb{Z}^d} g_n(\omega)\langle \nabla \rangle^s \eta(D-n)\phi\|_{L^p(\Omega)} \right\|_{L^2(\mathbb{R}^d)} \right )^p .\\
\end{split}
\end{equation}

By a well known lemma on sums of random variables, stated as Lemma 2.1 in \cite{Pocovnicu} and proven in \cite{Burq}, and the fact that Fourier multipliers commute, we have
\begin{equation}\label{22}
\begin{split} 
&\le C \left ( \left \| \sqrt{p}\| g_n\langle \nabla \rangle^s \eta(D-n)\phi\|_{l^2(n \in \mathbb{Z}^d)}\right \|_{L^2(\mathbb{R}^d)} \right )^p \\
\mathbb{E}[\|\phi^{\omega}\|^p_{H^s(\mathbb{R}^d)} ] &\le C \left (\sqrt{p}\|\phi\|_ {H^s(\mathbb{R}^d)}\right )^p .\\
\end{split}
\end{equation}

So by Markov's Inequality
\begin{equation}\label{25}
\begin{split}
 R^p P \left ( \|\phi^{\omega}\|_{H^s(\mathbb{R}^d)} > R \right ) &\le C \left (\sqrt{p}\|\phi\|_ {H^s(\mathbb{R}^d)}\right )^p \\
P \left ( \|\phi^{\omega}\|_{H^s(\mathbb{R}^d)} > R \right ) &\le \frac{ \left (C_0\sqrt{p}\|\phi\|_ {H^s(\mathbb{R}^d)}\right )^p}{R^p} .\\
\end{split}
\end{equation}
Now let $p= \left ( \frac{R}{C_0e\|\phi\|_{H^s}} \right )^2$ with $C_0$ taken from above. There are two cases.

\begin{itemize}
\item{} $p < 2$: In this case we cannot use the above work becuase it assumes $p \ge 2$ for Minkowski's inequality. 
Letting $c_2=\frac{1}{C_0^2e^2}$ we have $e^{-c_2R^2/\|\phi\|_{H^s}} \ge e^{-2}$. Now choosing $c_1 \ge e^2$ we have
\begin{equation}
\begin{split}
c_1e^{-c_2R^2/\|\phi\|_{H^s}^2} &\ge c_1e^{-2} \\
c_1e^{-c_2R^2/\|\phi\|_{H^s}^2} &\ge 1 \\
c_1e^{-c_2R^2/\|\phi\|_{H^s}^2} &\ge P\left ( \|\phi^{\omega}\|_{H^s(\mathbb{R}^d)} > R \right ), \\
\end{split}
\end{equation}
since every probabilistic outcome has probability less than $1$.

\item{} $p \ge 2$: From the definition above and equation (\ref{25}), we have
\begin{equation}
\begin{split}
P \left ( \|\phi^{\omega}\|_{H^s(\mathbb{R}^d)} > R \right ) &\le e^{-p} \\
P \left ( \|\phi^{\omega}\|_{H^s(\mathbb{R}^d)} > R \right ) &\le e^{-c_2R^2/\|\phi\|^2_{H^s}} .\\
\end{split}
\end{equation}
\end{itemize}
In both cases the lemma is proven.
\end{proof}

\begin{lemma} \label{thm2.2}
Given $\phi \in H^s$ with randomization $\phi^{\omega}$, for all $R>0$ there exist positive constants $c_1, c_2$ such that: 
\begin{equation}
\label{eq:8} P\left ( \|S(t)\phi^{\omega}\|_{L^q_tL^r_x([0,T] \times \mathbb{R}^d)} > R \right ) \le c_1e^{-c_2R^2/T^{2/q}\|\phi\|^2_{L^2}}. \\
\end{equation} 
\end{lemma}
After multiplying $R$ by a small power of $T$ we have the following corollary. 
\begin{corollary}
For small $\theta \in [0, \frac{1}{q})$ and $R>0$ there exists $c_2,c_2$ such that:
\begin{equation}
\begin{split}
\label{eq:8} P\left ( \|S(t)\phi^{\omega}\|_{L^q_tL^r_x([0,T] \times \mathbb{R}^d)} > T^{\theta}R \right ) &\le c_1e^{-c_2R^2/T^{2/q-2\theta}\|\phi\|^2_{L^2}} \\
&\le c_1e^{-c_2R^2/\|\phi\|^2_{L^2}} .\\
\end{split}
\end{equation} 
\end{corollary}
In addition,  placing derivatives inside and noting that derivatives commute with fourier multipliers such as $S(t)$ and the map $\phi \rightarrow \phi^{\omega}$, we have our main bound:
\begin{theorem} \label{zbound}
Given small $\theta \in [0, \frac{1}{q})$ and $\phi^{\omega}$ chosen according to a Wiener randomization, for all $R>0$ there exists $c_2,c_2$ such that:
\begin{equation}
\begin{split}
\label{eq:8} P\left (\|\langle \nabla \rangle^{s} S(t)\phi^{\omega}\|_{L^q_tL^r_x([0,T] \times \mathbb{R}^d)} > T^{\theta}R \right ) &\le c_1e^{-c_2R^2/T^{2/q-2\theta}\|\phi\|^2_{H^s}} \\
&\le c_1e^{-c_2R^2/\|\phi\|^2_{H^s}} .\\
\end{split}
\end{equation} 
\end{theorem}
This bound will be crucial in the proof of local well-posedness. This gives us as much integrability as we want in bounding a linear solution, which means the linear solution is bounded in subcritical norms, allowing us to treat local well-posedness like a subcritical problem.

\begin{lemma}  \label{thm2.3}
Given $\phi \in H^s$ with randomization $\phi^{\omega}$, for all $R>0$ there exist positive constants $c_1, c_2$ such that: 
\begin{equation}
\label{eq:8} P\left ( \|\phi^{\omega}\|_{L^p(\mathbb{R}^d)} > R \right ) \le c_1e^{-c_2R^2/\|\phi\|^2_{L^2}}. \\
\end{equation} 
\end{lemma}
\begin{proof}
The proofs can be found in \cite{Pocovnicu}. They utilize the same basic argument as above, with some extra steps. Each proof exploits an improved Bernstein's inequality that results from the Wiener randomization. Note that $g_n(\omega)\eta(D-n)\phi$ has Fourier transform supported on the unit cube centered at $n$. Therefore $e^{inx} g_n(\omega)\eta(D-n)\phi$ has Fourier transform supported on the unit cube centered at the origin. Bernstein's inequality implies that 
\begin{equation}
\|e^{inx} g_n(\omega)\eta(D-n)\phi\|_{L^p} \lesssim \|e^{inx} g_n(\omega)\eta(D-n)\phi\|_{L^2} \\
\end{equation}
with no loss of regularity, since multiplying by $e^{inx}$ does not affect the $L^p$ norm, so we obtain the bound $\|g_n(\omega)\eta(D-n)\phi\|_{L^p} \lesssim \|g_n(\omega)\eta(D-n)\phi\|_{L^2}$. This is the key ingredient in the proof that allows one to bound the higher $L^p$ norm of $\phi^{\omega}$ with high probability while only assuming that $\phi \in L^2$.
\end{proof}

\section{Littlewood Paley theory and Function Spaces}  \label{mathrefs}

\subsection{Littlewood Paley Theory and Dyadic Decompositions}
In the fixed point proof we will take the linear and non-linear parts of our solution and dyadically decompose each into a sum of Littlewood Paley projections. Given a smooth bump function $\psi$ such that $\psi(\xi) = 1$ for $|\xi| \le 1 $ and $\psi(\xi) = 0$ for $|\xi| \ge 2$ we have the following definition from the Littlewood Paley theory:
\begin{definition}
Given dyadic $N$ and a function $f \in L^2$ we define its projection $P_{\le N}f$ to be the Fourier multiplier such that $\widehat{P_{\le N}f}(\xi )= \psi(\frac{\xi}{N})\widehat{f}(\xi)$.
\end{definition}

Of course the definition applies to a much wider range of distributions, but in this paper we need only consider functions in $L^2$ or $H^s$ for some $s >0$.

Note that $\widehat{P_{\le N}f}$ is supported on the set $|\xi| \le 2N$. Now we define the projection $P_N$ that localizes to frequencies in the interval $[N/2, 2N]$.
\begin{definition}
We define $P_{1} = P_{\le 1}$ and for dyadic $N>1$, $P_{N}f = P_{ \le N}f - P_{\le N/2}f$.
\end{definition}

This defines the projection $P_Nf$ with frequencies between $N/2$ and $2N$. Also we have $\sum_{N} P_Nf = f$, so this is indeed a decomposition.

The above info and other results on Littlewood-Paley theory can be found in the appendix of \cite{Dispersive}.

\subsection{Strichartz Spaces}  \label{mathrefs}
In this and the following section we introduce the function spaces needed to prove well-posedness. We start with the standard Strichartz spaces: $S^{s}(I \times \mathbb{R}^d)$ and $N^s(I \times \mathbb{R}^d)$: Let $q,r$ be a Schr\"odinger-admissible pair. Given an interval $I = [t_0,t_1]$ we define $S^s(I \times \mathbb{R}^d)$ to be the set of measurable functions bounded in the following norm:
\begin{equation*}
\begin{split}
\|u\|_{S^{s}(I \times \mathbb{R}^d)} &= \sup_{(q,r) - \text{admissible}} \|\langle \nabla \rangle^s u\|_{L^qL^r(I \times \mathbb{R}^d)}.\\
\end{split}
\end{equation*}
We also define $N^{-s}(I \times \mathbb{R}^d)$ to be the dual space of $S^s(I \times \mathbb{R}^d)$, which satisfies the bound:

\begin{equation*}
\begin{split}
\|u\|_{N^{s}(I \times \mathbb{R}^d)} &\le \inf_{(q,r) - \text{admissible}} \|\langle \nabla \rangle^s u\|_{L^{q'}L^{r'}(I \times \mathbb{R}^d)}.\\
\end{split}
\end{equation*}

The key relation between the Strichartz norms is the Strichartz estimate for solutions to the non-linear Schr\"odinger equation. Suppose $u$ is a solution to $iu_t + \Delta u = F$, then 
\begin{equation} \label{300}
\|u\|_{S^s([t_0, t_1] \times \mathbb{R}^d)} \lesssim \|u(t_0)\|_{H^s(\mathbb{R}^d)} + \|F\|_{N^s([t_0, t_1] \times \mathbb{R}^d)}.
\end{equation}

\subsection{$U^p$ and $V^p$ spaces}
Now it turns out we will want to use a norm that measures how close a function is to a linear solution to the Schr\"odinger equation. We start by defining a $U^p$ atom, and then the $U^p$ and $V^p$ spaces. 
Suppose $ 1 \le p < \infty$ and $- \infty < t_0 < t_1 < \ldots, < t_n \le \infty$ is a partition of the real line. We will denote the characteristic function of the $k$th interval of this partition by $\chi_{[t_{k-1},t_k)}$. 

\begin{definition} 
A $U^p$ atom is a step function into some Sobolev space $a(t) : \mathbb{R} \rightarrow H^s(\mathbb{R}^d)$ of the form 
\begin{equation}
a = \sum_{k=1}^{n} \phi_{k}\chi_{[t_{k-1},t_k)}
\end{equation}

where $\displaystyle \sum_{k=1}^{n} \|\phi_{k}\|^p_{H^s(\mathbb{R}^d)} = 1$.
\end{definition}
The definition applies to any Hilbert space $H$, but we will only need it for Sobolev spaces in this paper. 
\begin{definition} The space $U^p(\mathbb{R};H^s)$ is the set of measurable functions bonuded in the associated norm:
\begin{equation}
\|u\|_{U^p(\mathbb{R};H^s)} = \inf_{U^p \text{ atoms } a_j}  \{ \sum_{j} |\lambda_j| : u = \sum_{j} \lambda_ja_j\} .
\end{equation}
\end{definition}
For the $V^p$ spaces we continue to partition the real line, and take our norm to be the \newline $p$-variation of the given function.
\begin{definition}
The space $V^p(\mathbb{R};H^s)$ is the set of functions bounded under the $V^p$ norm:
\begin{equation}
\|u\|_{V^p(\mathbb{R};H^s)} = \sup_{ \text{partitions } t_k} \left( \sum_{k=1}^{n} \|u(t_k)-u(t_{k-1})\|_{H^s(\mathbb{R}^d)}^p . \right)^{1/p} 
\end{equation}
\end{definition}  

In addition, given an interval $I$, the norms $\|u\|_{U^p(I;H^s)}, \|u\|_{V^p(I;H^s)}$ and any of the following norms are defined as the restriction norms, for example:
\begin{equation}
\|u\|_{U^p(I;H^s)} = \inf_{w(t)=u(t), t \in I, w(\infty) = 0 = w(-\infty)} \|w\|_{U^p(\mathbb{R};H^s)} .
\end{equation}

Now we want to create a norm that measures how close our function is to a linear solution to the Schr\"odinger equation, much like in the definition of the $X^{s,b}$ spaces. If $u$ is a linear solution then $S(-t)u$ is a function that is constant in time with $\|S(-t)u\|_{U^2(I;H^s)}$ and $\|S(-t)u\|_{V^2(I;H^s)}$ norms bounded by $\|u\|_{H^s}$.
We define the $U^p_{\Delta}H^s, V^p_{\Delta}H^s$ norms as 
\begin{equation*}
\begin{split}
\|u\|_{U^p_{\Delta}H^s(\mathbb{R};H^s)} &= \|S(-t)u\|_{U^p(\mathbb{R};H^s)} \\
\|u\|_{V^p_{\Delta}H^s(\mathbb{R};H^s)} &= \|S(-t)u\|_{V^p(\mathbb{R};H^s)} \\
\end{split}
\end{equation*}
and the spaces $U^p_{\Delta}H^s, V^p_{\Delta}H^s$ are defined as the set of measurable functions $u: \mathbb{R} \rightarrow H^s(\mathbb{R}^d)$ bounded in the $U^p_{\Delta}H^s$ and $V^p_{\Delta}H^s$ norms respectively. 
These are useful spaces, however, in our proof we will rely on dyadic decomposition and will need to apply these norms at specific frequencies, so it is more useful to do computations in a slightly different norm adapted to dyadic decompositions.

\begin{definition}
We define the $X^s$ and $Y^s$ norms, and associated spaces, as follows:
\begin{equation}
\begin{split}
\|u\|_{X^s(\mathbb{R})} &= \left (  \sum_{N} N^{2s}\|P_N u\|^2_{U^2_{\Delta}L^2}   \right)^{\frac{1}{2}}\\
\|u\|_{Y^s(\mathbb{R})} &= \left (  \sum_{N} N^{2s}\|P_N u\|^2_{V^2_{\Delta}L^2}   \right)^{\frac{1}{2}} .\\
\end{split}
\end{equation}
\end{definition}
Note that these norms are a little stronger than those above. They bound the closeness of the function $u$ to a solution to the linear equation at each frequency, not just generally. Note that we immediately have the embedding $X^s \hookrightarrow Y^s$ as well as the bound
$\|S(t)\phi\|_{X^s(\mathbb{R};H^s)} \le \|\phi\|_{H^s(\mathbb{R}^d)}$. This bound means that these spaces are well suited to studying the linear problem. 

In addition we define the following norm for the non-homogeneous term that will allow us to exploit duality:
\begin{equation}
\|F\|_{M^s(I)} = \left \| \int_{t_0}^{t} S(t-t')F(t')dt'  \right\|_{X^s(I)} .
\end{equation}

This is equivalent to the dual norm of $Y^s$, and we have the bound
\begin{equation}
\|F\|_{M^s(I)} \le \sup_{\|v\|_{Y^s(I)}=1} \int_{I}\int_{\mathbb{R}^d} F(t,x)v(t,x)dxdt
\end{equation} as Lemma 3.5 in \cite{Pocovnicu2}. This is equivalent to 
\begin{equation} \label{eqn:39}
\|F\|_{M^s(I)} \le \sup_{\|v\|_{Y^0(I)}=1} \int_{I}\int_{\mathbb{R}^d} \langle \nabla\rangle^{s} F(t,x)v(x,t)dxdt .
\end{equation} 

In addition we have a bound analogous to the Strichartz estimate (\ref{300}) for the $M^s$ norm. Suppose $u(t,x)$ is a solution to equation the Cauchy problem
\begin{equation}
\begin{split}
iu_t + \Delta u &= F \\
u|_{t=0} &= u(0) \\
\end{split}
\end{equation}

on the interval $I$. Then we have the bound
\begin{equation}
\|u\|_{X^s(I)} \lesssim \|u(0)\|_{H^s(\mathbb{R}^d)} + \|F\|_{M^s(I)} .
\end{equation}

\section{Strichartz Estimates}  \label{mathrefs}

\begin{lemma}  \label{Y0}
\begin{enumerate}

Let $q,r$ be a Schr\"odinger-admissible pair. 
\item{}
Given an interval $I$, for any $u \in Y^0(I)$ we have:
\begin{equation}
\|u\|_{L^q_tL^r_x(I \times \mathbb{R}^d)} \lesssim \|u\|_{Y^0(I)} .
\end{equation}
\item{}
Given an interval $I$ and $p \ge \frac{2(d +2)}{d}$, for any $u \in {Y}^{d/2 - (d+2)/p}(I)$ we have: 
\begin{equation}
\begin{split}
\|u\|_{L^p_tL^p_x(I \times \mathbb{R}^d)} &\lesssim \left \| |\nabla|^{d/2 - (d+2)/p}u \right \|_{Y^0(I)} \\
&\lesssim  \left \| u \right \|_{Y^{d/2 - (d+2)/p}(I)} .\\
\end{split}
\end{equation}
\end{enumerate}
\end{lemma}

\begin{proof}
The proof of the first is in \cite{Pocovnicu2}

To prove the second note that for $\frac{1}{p} = \frac{1}{r} - \frac{k}{d}$, Sobolev embedding implies that

\begin{equation}
\|u\|_{L^p(\mathbb{R}^d)} \lesssim \left \| |\nabla|^{k}u \right \|_{L^r(\mathbb{R}^d)}.
\end{equation}
Then taking the $L^p_t(I)$ norm of both sides we have
\begin{equation}
\|u\|_{L^p_{t,x}(I \times \mathbb{R}^d)} \lesssim \left \| |\nabla|^{k}u \right \|_{L^p_tL^r_x(I \times \mathbb{R}^d)}.
\end{equation}
Then by part 1, we have for $\frac{2}{p} + \frac{d}{r} = \frac{d}{2}$
\begin{equation}
\left\| |\nabla|^{k}u \right\|_{L^p_tL^r_x(I \times \mathbb{R}^d)} \lesssim \left \| |\nabla|^{k}u \right \|_{Y^0(I)} .
\end{equation}
This proves the desired inequality in $\mathbb{R}^d$ with exponents that satisfy
$\frac{2}{p}+\frac{d}{r} = \frac{d}{2}$ and $\frac{1}{p} = \frac{1}{r} - \frac{k}{d}$. Substituting we get $k = \frac{d}{2} - \frac{d+2}{p}$.
\end{proof}

By selecting $q=r=\frac{2(d+2)}{d}$ and $p = 2(d+2)$, we obtain the following corollaries:
\begin{corollary}
For all $u \in {Y}^0(I)$ one has:
\begin{equation}
\|u\|_{L^{\frac{2(d+2)}{d}}_{t,x}(I \times \mathbb{R}^d)} \lesssim \|u\|_{{Y}^0(I)} .
\end{equation}
\end{corollary}
\begin{corollary}
For all $u \in{Y}^{\frac{d-1}{2}}(I)$ one has:
\begin{equation}
\|u\|_{L^{2(d+2)}_{t,x}(I \times \mathbb{R}^d)} \lesssim \left \| |\nabla|^{\frac{d-1}{2}} u \right \|_{{Y}^{0}(I)} \lesssim \left \| u \right \|_{{Y}^{\frac{d-1}{2}}(I)} .
\end{equation}
\end{corollary}

Lastly, the following is a bilinear projection lemma that gives an $L^2$ bound on the bilinear $L^2$ norm of projections at different frequencies from Bourgain  in \cite{Bourgain}, \cite{Ozawa}. In addition there is a version adapted to the Schr\"odinger equation from \cite{Visan}. 
\begin{lemma}
For dyadic $N_1 \le N_2$ and $\phi_1, \phi_2 \in L^2$ we have
\begin{equation}
\|P_{N_1}S(t)\phi_1P_{N_2}S(t)\phi_2\|_{L^2(I \times \mathbb{R}^d)} \lesssim N_1^{\frac{d-1}{2}}N_2^{\frac{-1}{2}}\|P_{N_1}\phi_1\|_{L^2(\mathbb{R}^d)}\|P_{N_2}\phi_2\|_{L^2(\mathbb{R}^d)} .
\end{equation}
\end{lemma}
\begin{corollary} \label{bilinear} For  $N_1 \le N_2$ and $u_1, u_2 \in Y^0(I)$ we have
\begin{equation}
\|P_{N_1}u_1P_{N_2}u_2\|_{L^2(I \times \mathbb{R}^d)} \lesssim N_1^{\frac{d-1}{2}-}N_2^{\frac{-1}{2}+}\|P_{N_1}u_1\|_{Y^0(I)}\|P_{N_2}u_2\|_{Y^0(I)} .
\end{equation}
\end{corollary}
\begin{proof}The proof is found in \cite{Pocovnicu2} as Lemma 3.5. 
\end{proof}
This will be a key ingredient in the proof of local well-posedness because it allows us to gain half a derivative from higher frequency terms. In addition we use the following three dimensional bilinear estimate that solely consists of Strichartz norms. 

\begin{theorem} \label{visanbound} For dyadic $N_1 \le N_2$ and any small $\delta > 0$ we have:  
\begin{equation}
\begin{split}
\|P_{N_1}u_1P_{N_2}u_2\|_{L^2(I \times \mathbb{R}^3)} &\lesssim N_1^{\frac{d-1}{2}-\delta}N_2^{\frac{-1}{2}+\delta}(\|P_{N_1}u_1(0)\|_{L^2(\mathbb{R}^3)} + \|(i\partial_t + \Delta)P_{N_2}u_2\|_{L^{3/2}L^{18/13}(I\times \mathbb{R}^3)})\\
& \times (\|P_{N_2}u_2(0)\|_{L^2(\mathbb{R}^3)}+\|(i\partial_t + \Delta)P_{N_2}u_2 \|_{L^{3/2}L^{18/13}(I\times \mathbb{R}^3)}) .\\
\end{split}
\end{equation}
\end{theorem}
\begin{proof}The proof is found in \cite{Visan} as Lemma 2.5. 
\end{proof}
This will be a key ingredient in the proof of Theorem 1.2 in section 6.

\section{Almost Sure Local Well-Posedness} \label{mathrefs}
We now begin the proof of Theorem \ref{thm1}. Given some $\phi \in H^s(\mathbb{R}^d)$  let $\phi^{\omega}$ be its Wiener randomization, and recall that $z(t) = S(t)\phi^{\omega}$ denotes the linear part of the NLS solution and $v(t)$ is the solution to equation (\ref{non}). 

Even though we do not have long term bounds on the $H^s$ norm of $v$, we know that $v(0)=0$. Exploiting our probabilistic bound on $z(t)$ in subcritical norms, we show that for $\rho \in (\frac{d-1}{2}, s+\frac{1}{2})$ the norm $\|v\|_{X^{\rho}((-T,T))}$ is bounded for  small enough time $T$. 

Our method will be a fixed point argument. We define 
\begin{equation} \label{eq:55}
\Gamma v(t) = \pm \int_{0}^t -iS(t-t')[|v+z|^4(v+z)](t')dt'
\end{equation}
and note that $v$ is a solution if and only if $\Gamma v = v$. We now prove the following proposition, which is the bulk of our fixed point argument. 
\begin{proposition} Assume $s$ and $\rho$ satisfies the bounds 
\begin{equation}
\frac{d}{2} > s + \frac{1}{2} > \rho > \frac{d-1}{2} .
\end{equation}
Given $\phi \in H^s(\mathbb{R}^d)$ with randomization $\phi^{\omega}$ there exists small $\theta>0$ such that for every $R>0$ and sufficiently small $T \ll1$, we have 
\begin{itemize}
\item{} $\|\Gamma v\|_{X^{\rho}} \lesssim T^{\theta}(\|v\|^5_{X^{\rho}([0,T))}+R^5)$ off a set of measure $c_1e^{-c_2R^2/\|\phi\|^2_{H^s}}$. 
\item{} $\|\Gamma v_1 - \Gamma v_2\|_{X^{\rho}([0,T))} \lesssim T^{\theta}(R^4+\|v_1\|^4_{X^{\rho}([0,T))}+ \|v_2\|^4_{X^{\rho}([0,T))})\|v_1-v_2\|_{X^{\rho}([0,T))}$ off a set of measure $c_1e^{-c_2R^2/\|\phi\|^2_{H^s}}$.
\end{itemize}
This stems from Theorem \ref{zbound}, which tells us that for $\theta < \frac{1}{q}$ we have
\begin{equation}
P(\|\langle \nabla \rangle^{s} z\|_{L^q_tL^r_x([0,T) \times \mathbb{R}^d)} \le T^{\theta}R) \ge 1-c_1e^{-c_2R^2/\|\phi\|^2_{H^{s}}},
\end{equation} 
which allows us to gain a factor of $T$.
\end{proposition}
\begin{proof} 
We only prove the first part, as the proof of the second is similar. For dyadic $N \ge 1$ define 
\begin{equation*}
\begin{split}
\Gamma_N(v) &= P_{\le N} \Gamma(v) \\
&= P_{\le N} \left(\pm \int_{0}^t -iS(t-t')[|v+z|^4(v+z)](t')dt' \right) \\
&=\pm \int_{0}^t -iS(t-t')P_{\le N}[|v+z|^4(v+z)(t')]dt' .\\
\end{split}
\end{equation*}
By (\ref{eqn:39}) we have
\begin{equation*} \label{eq:57}
\begin{split}
\|\Gamma_N v\|_{X^{\rho}} &= \|P_{\le N}[(v+z)|v+z|^4]\|_{M^{\rho}}\\
&\le \sup_{v_6 | \|v_6\|_{Y^{0} \le 1}} \int_0^T \int_{\mathbb{R}^d} \langle \nabla^{\rho} \rangle |v + z|^4(v+z)\overline{P_{\le N}v_6} dx dt .\\
\end{split}
\end{equation*}

Now noting that 
\begin{equation*}
\begin{split}
\|\Gamma v\|_{X^{\rho}} &= \lim_{N \rightarrow \infty} \|\Gamma_N v\|_{X^{\rho}} \\
&= \sup_{v_6 | \|v_6\|_{Y^{0} \le 1}} \int_0^T \int_{\mathbb{R}^d} \langle \nabla^{\rho} \rangle |v + z|^4(v+z)\overline{v_6} dx dt\\
\end{split}
\end{equation*}
it suffices to show that for small $\theta >0$  this integral is $\le CT^{\theta}(R^5+\|v\|^5_{X^{\rho}})\|v_6\|_{Y^{0}}$ off a set of measure $c_1e^{-c_2R^2/\|\phi\|^2_{H^s}}$. We do this by proving the bound

\begin{equation} \label{eq:59}
\int_0^T \int_{\mathbb{R}^d} \langle \nabla \rangle^{\rho}[|v+z|^4(v+z)\overline{v_6}dx dt \le CT^{\theta}(R^5+\|v\|^5_{X^{\rho}})\|v_6\|_{Y^0}
\end{equation}
via case by case analysis of terms of the form $\langle \nabla \rangle^{\rho}[w_1w_2w_3w_4w_5]v_6$ where each $w_i$ is either 
\newline $v_i=v$ or $z_i=z$ (or it's complex conjugate), and each is dyadically decomposed into \newline $\sum_{N_i \ge 1, \text{dyadic}}P_{N_i}v_i, \sum_{N_j \ge 1, \text{dyadic}}P_{N_j}z_j$. Dyadic decomposition allows us to assume the derivatives are placed on the highest frequency term, or split them between two comparably high frequency terms. Also we will just write $w_i$ instead of $P_{N_i}w_i$ as we sum over dyadic integers $N_i \ge 1$.

We split the cases into four main cases based on whether each $w_i$ is a $v_i$ or $z_i$, and which two terms have the highest frequencies:

\begin{itemize} 
\item{\bf{Case 1}} All five terms are $v$.

\item{\bf{Case 2}} At least one term is a $v$ and it has one of the two highest frequencies.

\item{\bf{Case 3}} The two highest frequencies are on $z$ terms. 

\item{\bf{Case 4}} The two highest frequencies are on a $z$ term and the $v_6$ term. 
\end{itemize}
 
These four cases are then divided into smaller subcases:

\begin{enumerate} 
\item{\bf{Case 1: $v_1v_2v_3v_4v_5v_6$}}

In this case all terms are $v$'s. We do not do dyadic decompositions, instead we cut the frequency space into 5 pieces based on which frequency is largest, and assume without loss of generality that $\xi_1$ is. We split into two cases, based on the value of $\rho$, which determines which exponents we can use in H\"older's inequality. 

\begin{itemize}
\item{\bf{1.a:}} $\rho < \frac{d}{2} -\frac{1}{4}$.

Noting that $\rho < \frac{d}{2} - \frac{1}{4}$, we apply H\"older's inequality with $t$ exponents \newline $\left (\frac{2(d+2)}{d(2d-4\rho-1)} ,\frac{4(d+2)}{d+2-d(d-2\rho)} \times 4, \frac{2(d+2)}{d} \right )$ and $x$ exponents  $\left (\frac{2(d+2)}{8\rho+4-3d)} ,\frac{2(d+2)}{d-2\rho} \times 4, \frac{2(d+2)}{d} \right) $ and Lemma \ref{Y0}:
\begin{equation*}
\begin{split} 
I &=\int_0^T \int_{\mathbb{R}^d} \langle \nabla \rangle^{\rho}v_1v_2v_3v_4v_5v_6 dxdt \\
&\le \|\langle \nabla \rangle^{\rho}v_1\|_{L^{\frac{2(d+2)}{d(2d-4\rho-1)}}L^{\frac{2(d+2)}{8\rho+4-3d}}} \|v_2\|_{L^{\frac{4(d+2)}{d+2-d(d-2\rho)}}L^{\frac{2(d+2)}{d-2\rho}}} \|v_3\|_{L^{\frac{4(d+2)}{d+2-d(d-2\rho)}}L^{\frac{2(d+2)}{d-2\rho}}}\\
&\times  \|v_4\|_{L^{\frac{4(d+2)}{d+2-d(d-2\rho)}}L^{\frac{2(d+2)}{d-2\rho}}} \|v_5\|_{L^{\frac{4(d+2)}{d+2-d(d-2\rho)}}L^{\frac{2(d+2)}{d-2\rho}}} \|v_6\|_{L^{\frac{2(d+2)}{d}}} \\
&\le \|v_1\|_{Y^{\rho}} T^{\theta} \|v_2\|_{L^{\frac{2(d+2)}{d-2\rho}}L^{\frac{2(d+2)}{d-2\rho}}} \|v_3\|_{L^{\frac{2(d+2)}{d-2\rho}}L^{\frac{2(d+2)}{d-2\rho}}} \|v_4\|_{L^{\frac{2(d+2)}{d-2\rho}}L^{\frac{2(d+2)}{d-2\rho}}} \|v_5\|_{L^{\frac{2(d+2)}{d-2\rho}}L^{\frac{2(d+2)}{d-2\rho}}} \|v_6\|_{Y^{0}}\\
&\le T^{\theta}\Pi_{i=1}^{5} \|v_i\|_{Y^{\rho}} \|v_6\|_{Y^{0}}\\
\end{split}
\end{equation*}
for some $\theta >0$.
\item{\bf{1.b:}} $\frac{d}{2} -\frac{1}{4} \le \rho < s + \frac{1}{2}$.

Noting that $\rho < \frac{d}{2} - \frac{1}{4}$ we apply H\"older's inequality with $t$ exponents \newline $\left (\infty ,\frac{8(d+2)}{d+4} \times 4, \frac{2(d+2)}{d} \right )$ and $x$ exponents  $\left (2,4(d+2) \times 4, \frac{2(d+2)}{d} \right) $ and Lemma \ref{Y0},
\begin{equation*}
\begin{split} 
I &=\int_0^T \int_{\mathbb{R}^d} \langle \nabla \rangle^{\rho}v_1v_2v_3v_4v_5v_6 dxdt \\
&\le \|\langle \nabla \rangle^{\rho}v_1\|_{L^{\infty}L^{2}} \|v_2\|_{L^{\frac{8(d+2)}{d+4}}L^{4(d+2)}} \|v_3\|_{L^{\frac{8(d+2)}{d+4}}L^{4(d+2)}}\\
&\times  \|v_4\|_{L^{\frac{8(d+2)}{d+4}}L^{4(d+2)}} \|v_5\|_{L^{\frac{8(d+2)}{d+4}}L^{4(d+2)}} \|v_6\|_{L^{\frac{2(d+2)}{d}}} \\
&\le \|v_1\|_{Y^{\rho}} T^{\theta} \|v_2\|_{L^{4(d+2)}}    \|v_3\|_{L^{4(d+2)}}  \|v_4\|_{L^{4(d+2)}} \|v_5\|_{L^{4(d+2)}}   \|v_6\|_{Y^{0}}\\
&\le T^{\theta}\|v_1\|_{Y^{\rho}} \Pi_{i=2}^{5} \|v_i\|_{Y^{\frac{d}{2} - \frac{1}{4}}} \|v_6\|_{Y^{0}}\\
\end{split}
\end{equation*}
for some $\theta >0$.
\end{itemize}
\item{\bf{Case 2: $v_1w_2w_3w_4z_5v_6$,  $N_1 \gtrsim N_2,N_3,N_4,N_5$}}

In this case there is at least one $v$ term and the highest frequency term is a $v$. Therefore we can assume the derivatives fall on the $v_1$ term with the highest frequency.
\begin{itemize}
\item{\bf{2.a:}} $w_2, w_3, w_4$ are all $z$ terms, $N_5 \ge N_4 \ge N_3 \ge N_2 \ge N_1^{1/2(d-1)}$  

We have assumed that $v_1$ has the highest frequency: $N_1 \ge N_2,N_3,N_4,N_5$. Now we apply H\"older's inequality, Lemma \ref{Y0} and our probabilistic bound on the linear term, Theorem \ref{zbound}, and note that $s > \frac{1}{2}$:  
\begin{equation*}
\begin{split} 
I &=\int_0^T \int_{\mathbb{R}^d} \langle \nabla \rangle^{\rho}v_1z_2z_3z_4z_5v_6 dxdt\\
&\le \|\langle \nabla \rangle^{\rho}v_1\|_{L^{2(d+2)/d}}\|z_2\|_{L^{2(d+2)}}\|z_3\|_{L^{2(d+2)}}\|z_4\|_{L^{2(d+2)}}\|z_5\|_{L^{2(d+2)}}\|v_6\|_{L^{2(d+2)/d}}\\
&\le \|v_1\|_{Y^{\rho}} (N_2N_3N_4N_5)^{-s}\Pi_{i=2}^{5} \| \langle \nabla \rangle^{s} z_i\|_{L^{2(d+2)}}\|v_6\|_{Y^{0}}\\
&\le \|v_1\|_{Y^{\rho}} (N_2N_3N_4N_5)^{\frac{-1}{2}}\Pi_{i=2}^{5} \|\langle \nabla \rangle^{s}  z_i\|_{L^{2(d+2)}}\|v_6\|_{Y^{0}}\\
&\le \|v_1\|_{Y^{\rho}} (N_2)^{-2}\Pi_{i=2}^{5} \|\langle \nabla \rangle^{s}  z_i\|_{L^{2(d+2)}}\|v_6\|_{Y^{0}}\\
&\le \|v_1\|_{Y^{\rho}} (N_1)^{\frac{-1}{d-1}}\Pi_{i=2}^{5} \| \langle \nabla \rangle^{s} z_i\|_{L^{2(d+2)}}\|v_6\|_{Y^{0}}.\\
\end{split}
\end{equation*}
Noting that $N_1$ is the highest frequnecy, the sum over all frequencies is bounded by $\|v\|_{Y^{\rho}}T^{\theta}R^{4}\|v_6\|_{Y^{0}}$ off a set of  measure $c_1e^{-c_2R^2/\|\phi\|^2_{H^s}}$.

\item{\bf{2.b:}} $w_2, w_3, w_4$ are all $z$ terms, $N_2 \le N_1^{1/2(d-1)}$

We apply H\"older's inequality, \ref{zbound}, \ref{Y0} and our bilinear estimate \ref{bilinear}, utilizing the assumption that $N_2 \le N_1^{1/2(d-1)}$: 
\begin{equation*}
\begin{split} 
I &=\int_0^T \int_{\mathbb{R}^d} \langle \nabla \rangle^{\rho}v_1z_2z_3z_4z_5v_6 dxdt\\
&\le \|\langle \nabla \rangle^{\rho}v_1 z_2\|_{L^{2}}\|z_3\|_{L^{3(d+2)}}\|z_4\|_{L^{3(d+2)}}\|z_5\|_{L^{3(d+2)}}\|v_6\|_{L^{2(d+2)/d}}\\
&\le N_1^{\frac{-1}{2}+}\|v_1\|_{Y^{\rho}} N_2^{\frac{d-1}{2}-}\|z_2\|_{Y^{0}}\|z_3\|_{L^{3(d+2)}}\|z_4\|_{L^{3(d+2)}}\|z_5\|_{L^{3(d+2)}}\|v_6\|_{L^{2(d+2)/d}}\\
&\le N_1^{\frac{-1}{2}+}\|v_1\|_{Y^{\rho}}N_1^{\frac{1}{4}}\|z_2\|_{Y^0}T^{0+}R^3\|v_6\|_{Y^0}\\
&\le N_1^{\frac{-1}{4}+}\|v_1\|_{Y^{\rho}}T^{0+}R^4\|v_6\|_{Y^0},\\
\end{split}
\end{equation*}
which is $\le\|v\|_{Y^{\rho}}T^{0+}R^4$ off a set of small measure.

\item{\bf{2.c:}} $w_2 =v_2$ is a $v$ term, and the others can be anything

In this case we still have $N_1 \ge N_i, i=2,\ldots,6$. Applying H\"older's inequality, \ref{Y0}, \ref{zbound} and \ref{bilinear}, we have
\begin{equation*}
\begin{split} 
I &=\int_0^T \int_{\mathbb{R}^d} \langle \nabla \rangle^{\rho}v_1w_2w_3w_4z_5v_6 dxdt\\
&\le \|\langle \nabla \rangle^{\rho}v_1 v_2\|_{L^{2}}\|w_3\|_{L^{2(d+2)}}\|w_4\|_{L^{4(d+2)}}\|z_5\|_{L^{4(d+2)}}\|v_6\|_{L^{2(d+2)/d}}\\
&\le N_1^{\frac{-1}{2}+}\|v_1\|_{Y^{\rho}}N_2^{\frac{d-1}{2}-}\|v_2\|_{Y^0}\|w_3\|_{L^{2(d+2)}}\|w_4\|_{L^{4(d+2)}}\|z_5\|_{L^{4(d+2)}}\|v_6\|_{Y^0}\\
&\le N_1^{\frac{-1}{2}+}\|v_1\|_{Y^{\rho}}\|v_2\|_{Y^{(d-1)/2}}\|w_3\|_{L^{2(d+2)}}\|w_4\|_{L^{4(d+2)}}\|z_5\|_{L^{4(d+2)}}\|v_6\|_{Y^0}.\\
\end{split}
\end{equation*}
Now if $w_3$ is a $v$ term, then $\|w_3\|_{L^{2(d+2)}} \lesssim \|v_3\|_{Y^{(d-1)/2}}$ as required. If $w_3$ is a $z$ term, then $\|w_3\|_{L^{2(d+2)}} \le T^{0+}R$ off a set of small measure. So either way this term is bounded. 

If $w_4$ is a $z$ term then, again, $\|w_4\|_{L^{4(d+2)}} \le T^{0+}R$ off a set of small measure. The only trouble is if $w_4$ is a $v$ term, in which case our inequality only gives us:
\begin{equation*}
\begin{split}
\|w_4\|_{L^{4(d+2)}} & \lesssim \left\|  |\nabla |^{\frac{d}{2} -\frac{1}{4}}v_4 \right\|_{Y^0} \\
&\lesssim N_4^{1/4}\|v_4\|_{Y^{(d-1)/2}} .\\
\end{split}
\end{equation*}We have an extra quarter derivative, however, since $N_1$ is the biggest frequency we have $N_4^{1/4} \le N_1^{1/4}$, which is absorbed by the $N_1^{\frac{-1}{2}+}$ term.

Therefore each term in this case is bounded by $T^{0+}(\|v\|_{X^{\rho}}^5+ R^5)\|v_6\|_{Y^0}$ off a set of measure $c_1e^{-c_2R^2/\|\phi\|^2_{H^s}}$.

\end{itemize}

\item{\bf{Case 3:} $w_1w_2w_3z_4z_5v_6$}, $N_4 \sim N_5 \gtrsim N_1,N_2,N_3, N_6$.

In this case the two biggest frequencies are on $z$ terms, $z_4$ and $z_5$. 
The first three terms are denoted $w_i, i=1,2,3$ and represent either $v$ or $z$. Assume without loss of generality that $N_1 \le N_2 \ldots \le N_4 \sim N_5 \ge N_6$. Applying H\"older's inequality for exponents $\left( 2(d+2) \times 3, \frac{4(d+2)}{d+1} \times 2, \frac{2(d+2)}{d} \right)$, \ref{Y0}, and \ref{zbound} we have
\begin{equation*}
\begin{split} 
\int_0^T \int_{\mathbb{R}^d} w_1w_2w_3z_4\langle \nabla \rangle^{\rho}z_5v_6 dxdt &\le \|w_1\|_{L^{2(d+2)}} \|w_2\|_{L^{2(d+2)}} \|w_3\|_{L^{2(d+2)}} \\
&\times \|\langle \nabla \rangle^{\rho/2}z_4\|_{L^{\frac{4(d+2)}{d+1}}}\|\langle \nabla \rangle^{\rho/2}z_5\|_{L^{\frac{4(d+2)}{d+1}}} \|v_6\|_{L^{\frac{2(d+2)}{d}}} .\\
\end{split}
\end{equation*}
For $\frac{\rho}{2} < s$ this term is bounded by $T^{0+}R^2(\|v\|_{Y^{(d-1)/2}}^3+T^{0+}R^3)\|v_6\|_{Y^{0}}$ off a set of measure $c_1e^{-c_2R^2/\|\phi\|^2_{H^s}}$. Note that $\rho < s+\frac{1}{2}<2s$ so $\rho$ satisfies the requirements. 

\item{\bf{Case 4:} $w_1w_2w_3w_4z_5v_6$}, $N_5 \sim N_6 \gtrsim N_1,N_2,N_3, N_4$.
This is the biggest case by far and we divide it into several subcases based on how many $v's$ there are.

\begin{itemize}
\item{\bf{4.a}} $z_1z_2z_3z_4z_5v_6$, $N_5 \sim N_6 \gtrsim N_1,N_2,N_3, N_4$.
Assume without loss of generality that $N_1 \le N_2 \ldots N_5 \sim N_6$. By H\"older's inequality with exponents $(2,8,8,8,8)$, \ref{bilinear} and \ref{zbound}, we have 
\begin{equation*}
\begin{split} 
I &= \int_0^T \int_{\mathbb{R}^d} z_1z_2z_3z_4\langle \nabla \rangle^{\rho}z_5v_6 dxdt \\
&\le \|z_1v_6\|_{L^2}\|z_2\|_{L^8}\|z_3\|_{L^8}\|z_4\|_{L^8}\|\langle \nabla \rangle^{\rho}z_5\|_{L^8} \\
&\lesssim N_1^{\frac{(d-1)}{2}-s+}N_5^{\frac{-1}{2}+}\|\langle \nabla \rangle ^s z_1\|_{Y^0}\|v_6\|_{Y^0}\|z_2\|_{L^8}\|z_3\|_{L^8}\|z_4\|_{L^8}\|\langle \nabla \rangle^{\rho}z_5\|_{L^8} \\
&\lesssim  N_1^{\frac{(d-1)}{2}+}(N_1N_2N_3N_4)^{-s}N_5^{\rho+\frac{-1}{2}-s+} \Pi_{i=2}^{5}\|\langle \nabla \rangle ^s z_i\|_{Y^0}\|\langle \nabla \rangle^{s}z_i\|_{L^8}\|v_6\|_{Y^0} .\\
\end{split}
\end{equation*}
When $s+\frac{1}{2}> \rho$ and $s > \frac{d-1}{8}$ the powers of the frequencies are negative and the sum is bounded by $T^{0+}R^5\|v_6\|_{Y^0}$. We have assumed $s+ \frac{1}{2} > \rho \ge \frac{d-1}{2}$ in the statement of the theorem, and note that for $d \ge 3$, $\frac{d-2}{2} > \frac{d-1}{8}$ and therefore we only require $s > \frac{d-2}{2}$, however, $s > \rho - \frac{1}{2} \ge \frac{d-2}{2}$. Therefore this term is bounded by $T^{0+}R^5\|v_6\|_{Y^{0}}$ off a set of measure $c_1e^{-c_2R^2/\|\phi\|^2_{H^s}}$. 

In all following cases, we can assume there is at least one $v$, at least one $z$, and that $N_5$ and $N_6$ are the highest frequencies.

\item{\bf{4.b}} $v_1z_2z_3z_4z_5v_6$, $N_5 \sim N_6 \gtrsim N_1,N_2,N_3, N_4$.

Assume without loss of generality that $N_2 \le  N_3 \ldots \le N_5 \sim N_6 \ge N_1$. 

Noting that $N_2 \le N_3,N_4$ we apply H\"older's inequality, \ref{bilinear} and \ref{zbound} to obtain
\begin{equation*}
\begin{split} 
I &= \int_0^T \int_{\mathbb{R}^d} v_1z_2z_3z_4\langle \nabla \rangle^{\sigma}z_5v_6 dxdt\\
 &\le \|z_2v_6\|_{L^2}\|v_1\|_{L^{2(d+2)}}\|z_3\|_{L^{6(d+2)/(d+1)}}\|z_4\|_{L^{6(d+2)/(d+1)}}\|\langle \nabla \rangle^{\sigma}z_5\|_{L^{6(d+2)/(d+1)}} \\
&\lesssim N_6^{\frac{-1}{2}+}N_2^{\frac{d-1}{2}-}\|z_2\|_{Y^0}\|v_6\|_{Y^0} \|v_1\|_{Y^{\frac{d-1}{2}}}(N_3N_4N_5)^{-s}R^3 \\
&\lesssim N_6^{\frac{-1}{2}+}N_2^{\frac{d-1}{2}-s-}(N_3N_4N_5)^{-s}T^{0+}R^4\|v_1\|_{Y^{\frac{d-1}{2}}}\|v_6\|_{Y^0} \\
&\lesssim N_6^{\frac{-1}{2}+}(N_2N_3N_4N_5)^{\frac{d-1}{8}-s}\|v_1\|_{Y^{\frac{d-1}{2}}}T^{0+}R^4\|v_6\|_{Y^0}.\\
\end{split}
\end{equation*}
When $s > \frac{d-1}{8}$ the powers of the frequencies are negative, and the sum is bounded by $\|v\|_{Y^{\frac{d-1}{2}}}T^{0+}R^4\|v_6\|_{Y^0}$. As demonstrated in the previous case, $s > \frac{d-1}{8}$.

\item{\bf{4.c}} $v_1v_2z_3z_4z_5v_6$, $ N_5 \sim N_6 \gtrsim N_1,N_2,N_3, N_4$.

Assume without loss of generality that $N_1 \ge N_2, N_3 \le N_4 \le N_5 \sim N_6$.
By H\"older's Inequality, \ref{bilinear} and \ref{zbound} we have 
\begin{equation*}
\begin{split} 
I &= \int_0^T \int_{\mathbb{R}^d} v_1v_2z_3z_4\langle \nabla \rangle^{\rho}z_5v_6 dxdt\\
I &\le \|v_1v_6\|_{L^2}\|v_2\|_{L^{2(d+2)}}\|z_3\|_{L^{\frac{6(d+2)}{d+1}}}\|z_4\|_{L^{\frac{6(d+2)}{d+1}}}\|\langle \nabla \rangle^{\rho}z_5\|_{L^{\frac{6(d+2)}{d+1}}} \\
&\lesssim N_6^{\frac{-1}{2}+}N_1^{\frac{d-1}{2}-}\|v_1\|_{Y^0}\|v_6\|_{Y^0} \|v_2\|_{Y^{\frac{d-1}{2}}}(N_3N_4)^{-s}(N_5)^{\rho}T^{0+}R^3\\
&\lesssim N_1^{0-} (N_3N_4)^{-s}N_5^{\rho-s-\frac{1}{2}+}\|v_1\|_{Y^{\frac{d-1}{2}}}\|v_2\|_{Y^{\frac{d-1}{2}}}T^{0+}R^3\|v_6\|_{Y^0} \\
&\lesssim N_1^{0-}N_2^{0-} (N_3N_4)^{-s}N_5^{\rho-s-\frac{1}{2}+}\|v_1\|_{Y^{\frac{d-1}{2}}}\|v_2\|_{Y^{\frac{d-1}{2}}}T^{0+}R^3\|v_6\|_{Y^0} .\\
\end{split}
\end{equation*}
Since $s +\frac{1}{2}< \rho$, all powers of the frequencies are negative and the sum is bounded by $\|v\|^2_{Y^{\frac{d-1}{2}}}T^{0+}R^3\|v_6\|_{Y^0}$. 

\item{\bf{4.d}} $v_1v_2v_3z_4z_5v_6$, $N_5 \sim N_6 \gtrsim N_1,N_2,N_3, N_4$.

Assume without loss of generality that $N_1 \ge N_2 \ge N_3, N_4 \le N_5$.
By H\"older's Inequality, \ref{bilinear} and \ref{zbound}  we have
\begin{equation*}
\begin{split} 
I &= \int_0^T \int_{\mathbb{R}^d} v_1v_2v_3z_4\langle \nabla \rangle^{\rho}z_5v_6 dxdt\\
I &\le \|v_1v_6\|_{L^2}\|v_2\|_{L^{2(d+2)}}\|v_3\|_{L^{2(d+2)}}\|z_4\|_{L^{\frac{4(d+2)}{d}}}\|\langle \nabla \rangle^{\rho}z_5\|_{L^{\frac{4(d+2)}{d}}} \\
&\lesssim  N_6^{\frac{-1}{2}+}N_1^{\frac{d-1}{2}-}\|v_1\|_{Y^0}\|v_6\|_{Y^0} \|v_2\|_{Y^{\frac{d-1}{2}}} \|v_3\|_{Y^{\frac{d-1}{2}}} N_4^{-s}N_5^{\rho - s}T^{0+}R^2\\
&\lesssim N_1^{0-}N_4^{-s}N_5^{\rho-s-\frac{1}{2}+}\|v_1\|_{Y^{\frac{d-1}{2}}}\|v_2\|_{Y^{\frac{d-1}{2}}} \|v_3\|_{Y^{\frac{d-1}{2}}}T^{0+}R^2\|v_6\|_{Y^0} .\\
\end{split}
\end{equation*}
Since $s +\frac{1}{2}< \rho$, all powers of the frequencies are negative and the sum is bounded by $\|v\|^3_{Y^{\frac{d-1}{2}}}T^{0+}R^2\|v_6\|_{Y^0}$. 

\item{\bf{4.e}} $v_1v_2v_3v_4z_5v_6$, $N_5 \sim N_6 \gtrsim N_1,N_2,N_3, N_4$.

Assume without loss of generality that $N_1 \le N_2 \le N_3 \le N_4 \le N_5$.
By H\"older's Inequality, \ref{bilinear} and \ref{zbound}  we have
\begin{equation*}
\begin{split} 
I &= \int_0^T \int_{\mathbb{R}^d} v_1v_2v_3v_4\langle \nabla \rangle^{\rho}z_5v_6 dxdt\\
I &\le \|v_1v_6\|_{L^2}\|v_2\|_{L^{2(d+2)}}\|v_3\|_{L^{2(d+2)}}\|v_4\|_{L^{\frac{4(d+2)}{d}}}\|\langle \nabla \rangle^{\rho}z_5\|_{L^{\frac{4(d+2)}{d}}} \\
&\lesssim  N_6^{\frac{-1}{2}+}N_1^{\frac{d-1}{2}-}\|v_1\|_{Y^0}\|v_6\|_{Y^0} \|v_2\|_{Y^{\frac{d-1}{2}}} \|v_3\|_{Y^{\frac{d-1}{2}}} \|v_4\|_{Y^{\frac{d}{4}}}N_5^{\rho - s}T^{0+}R\\
&\lesssim N_1^{0-}N_4^{\frac{-(d-2)}{4}}N_5^{\rho-s-\frac{1}{2}+}\|v_1\|_{Y^{\frac{d-1}{2}}}\|v_2\|_{Y^{\frac{d-1}{2}}} \|v_3\|_{Y^{\frac{d-1}{2}}}\|v_4\|_{Y^{\frac{d-1}{2}}}T^{0+}R\|v_6\|_{Y^0} .\\
\end{split}
\end{equation*}
Since $s +\frac{1}{2}< \rho$, all powers of the frequencies are negative and the sum is bounded by $\|v\|_{Y^{\frac{d-1}{2}}}^4T^{0+}R\|v_6\|_{Y^{0}}$ off a set of small measure.
\end{itemize}
\end{enumerate}

In each case the term is bounded by, for some $\theta >0$, $CT^{\theta}(R^5+\|v\|^5_{Y^{\rho}})\|v_6\|_{Y^{0}}$.  
This completes the proof of the first part of the proposition.
\newline The proof of $\|\Gamma v_1 - \Gamma v_2\|_{X^{\rho}} \le CT^{\theta}(R^4+\|v_1\|^4_{X^{\rho}} + \|v_2\|^4_{X^{\rho}})\|v_1-v_2\|_{X^{\rho}}$  off a set of measure $c_1e^{-c_2R^2/\|\phi\|^2_{H^s}}$ is similar and is omitted. 
\end{proof}
Using this key proposition we can close the fixed point argument in the final theorem. 

\textit{Proof of Theorem \ref{thm1}}
\newline Let $B_r$ be the ball of radius $r$ in $X^{\rho}([0,T))$ with $\frac{d}{2} > s+\frac{1}{2} > \rho \ge \frac{d-1}{2}$ as in the previous proposition. I claim that for small enough $T$ and small but fixed $r$ the map $\Gamma$ is a contraction on $B_r$ outside a set of measure $c_1e^{-c_2R^2/\|\phi\|^2_{H^s}}$. See section 1.6 of \cite{Dispersive} for an overview of contraction based fixed point arguments. 

To apply the theory for fix point arguments we require, off a small set, the contraction conditions
\begin{itemize}
\item{} $\|\Gamma v\|_{X^{\rho}([0,T))} \le r$ for $v \in B_r$ 
\item{} $\|\Gamma v_1 -\Gamma v_2\|_{X^{\rho}([0,T))} \le \frac{1}{2}\|v_1 - v_2\|_{X^{\rho}([0,T))}.$
\end{itemize}
By the bounds from the proposition, we have for all $R$ and some fixed constant $C$, \newline $\|\Gamma v\|_{X^{\rho}} \le CT^{\theta}(R^5+r^5)$ and $\|\Gamma v_1 -\Gamma v_2\|_{X^{\rho}} \le C\|v_1 - v_2\|_{X^{\rho}}T^{\theta}(2r^4+R^4)$ off a set of measure $c_1e^{-c_2R^2/\|\phi\|_{H^{s}}}$. 

The contraction conditions are satisfied if we select $r,R,T$ such that 
\begin{equation}
\begin{split}
r &\le R\\
CT^{\theta}R^5 &\le r/8\\
\end{split}
\end{equation} 
We can fix a value of $r$ to satisfy the first bound. Selecting $T$ such that $T \sim R^{\frac{-5}{\theta}}$ the second bound of a contraction is satisfied, and we conclude that the map $\Gamma$ has a fixed point in $B_r$.

Therefore for sufficiently small $T$, the equation $\Gamma v = v$ has a solution in $B$ for every $\phi^{\omega}$ off this set of measure $c_1e^{-c_2R^2/\|\phi\|^2_{H^{s}}}$. Setting $\alpha =\frac{-2\theta}{5}$ there exists a set $\Omega_T \subset \Omega$ of measure $\ge 1-c_1e^{-c_2/T^{\alpha}\|\phi\|^2_{H^{s}}}$ such that for $t \in [0,T)$ the Duhamel equation 
\begin{equation}
v(t) = \pm \int_{0}^t -iS(t-t')[|v+z|^4(v+z)](t')dt'
\end{equation}
has a unique solution in $X^{\rho}([0,T))$. The same argument proves the existence of a solution in $X^{\rho}((-T,0])$ on a set of the same measure. Taking $u(t) = S(t)\phi+ v(t)$ we have a solution on the interval $(-T,T)$ in the class $H^s(\mathbb{R}^d) + C((-T,T) \rightarrow H^{\rho}(\mathbb{R}^3)) \subset H^s(\mathbb{R}^d)$.

\section{A Condition for Global Well-Posedness}  \label{mathrefs}
We now present the proof of Theorem 1.2. The proof relies upon the following proposition. 

\begin{proposition} \label{60}
Suppose $0 < c < \frac{1}{8}$, $\frac{7}{8} < s < 1$ and $\|\phi^{\omega}\|_{H^s(\mathbb{R}^3)} < R$. There exists a small positive constant  $\epsilon \ll \frac{1}{R}$ such that for any interval $[t_1,t_2]$ satisfying, $|t_1-t_2|\le 1$, $\|v\|_{L^{10}L^{10}([t_1,t_2] \times \mathbb{R}^3)} < \epsilon$ and $\|\langle \nabla \rangle ^{s} z\|_{L^qL^r([t_1,t_2]\times \mathbb{R}^3)} < \epsilon$ for the pairs \newline $(q,r) \in  \{(10,10), (15/2,15/7),(30/7,15)\}$, we have $\|v\|_{S^{1+c}([t_1,t_2]\times \mathbb{R}^3)} \lesssim \|v(t_1)\|_{H^{1+c}(\mathbb{R}^3)} + C(\epsilon)$. 
\end{proposition}

We first give the proof of Theorem 1.2 given that Proposition 6.1 is true. The rest of the paper is devoted to proving Proposition 6.1.

\textit{Proof of Theorem 1.2}
Assume Proposition \ref{60} and the hypothesis of Theorem 1.2, that there exists such a function $\alpha$. Fix values $T,R$ and a set $\Omega_{T,R}'$ satisfying the properties outlined in Theorem 1.2.
By Theorem \ref{zbound} and Theorem \ref{thm2.1} there is a set $\Omega_{T,R} \subset \Omega'_{T,R} $ of measure at least $1- c_1e^{-c_2R^2/T\|\phi\|_{H^s}} -\alpha(T,R)$ such that for any $\omega \in \Omega_{T,R}$ and \newline $(q,r) \in  \{(\infty,2),(10,10), (15/2,15/7),(30/7,15)\}$ we have
\begin{equation}
\|\langle \nabla \rangle^s z\|_{L^qL^r([-T,T]\times \mathbb{R}^3)} < R,
\end{equation}
 and for any solution $v$ to (\ref{non}) we have
\begin{equation}
\|v\|_{L^{10}L^{10}([-T,T] \times \mathbb{R}^3)} <R.
\end{equation}

Now assume that $\omega$ is indeed in the set $\Omega_{T,R}$. Note that by the local well-posedness theory a solution exists on some short time interval $(-t,t)$. Suppose for sake of contradiction there is a pair of times $-T < T_{\min} < 0 < T_{\max} < T$ such that the solution $v(t)$ cannot be extended in $H^{1+c}$ past $(T_{\min},T_{\max})$. 

We know that $\|v\|_{L^{10}L^{10}((T_{\min},T_{\max}) \times \mathbb{R}^3)} <R$ and we have $\|\langle \nabla \rangle^s z\|_{L^qL^r([-T,T]\times \mathbb{R}^3)}< R$ for each necessary pair $(q,r)$, therefore  we can split $[T_{\min},T_{\max}]$ into a finite number of subintervals $I$ on which $\|v\|_{L^{10}L^{10}(I \times \mathbb{R}^3)}< \epsilon$ and $\|\langle \nabla \rangle^s z\|_{L^qL^r(I\times \mathbb{R}^3)} < \epsilon$ for \newline $(q,r) \in  \{(10,10), (15/2,15/7),(30/7,15)\}$.

This means that on each subinterval $[t_i,t_{i+1}]$ the conditions of Proposition 6.1 are met, and therefore the $\|v\|_{{S}^{1+c}([t_i,t_{i+1}] \times \mathbb{R}^3)}$ norm is finite. Therefore there exists a solution in the space $S^{1+c}([t_i,t_{i+1}] \times \mathbb{R}^3)$ on each succesive interval $[t_i,t_{i+1}]$ which implies that the $\|v\|_{L^{\infty}{H}^{1+c}}$ norm is bounded at each endpoint. This  means the ${S}^{1+c}$ norm is bounded on the next interval. Iterating this argument over each subinterval this implies the ${S}^{1+c}$ norm of the nonlinear solution $v(t)$ is bounded on the whole interval $[T_{\min},T_{\max}]$. In addition $\|v(T_{\min})\|_{H^{1+c}}$ and $\|v(T_{\max})\|_{{H}^{1+c}}$ are both finite. Therefore one can apply the local wellposedness theory to extend the solution beyond $[T_{\min},T_{\max}]$, which is a contradiction.

This concludes the proof of Theorem 1.2. It remains to prove Proposition 6.1

\textit{Proof of Proposition \ref{60}:} The nonlinear part of the solution $v$ satisfies the differential equation
\begin{equation} \label{eqn:61}
\begin{split}
iv_t + \Delta v &=  (v+z)|v+z|^4 \\
iv_t + \Delta v &=  v|v|^4 + f(v,z) \\
\end{split}
\end{equation}
for the function $f(v,z) =  (v+z)|v+z|^4 -v|v|^4 \lesssim |z|^5 + |z|\cdot |v|^4$.  

By the Strichartz estimates, (\ref{300}), we have the bound
\begin{equation}
\|v\|_{{S}^{1+c}([t_1,t_2]\times \mathbb{R}^3)} \lesssim \|v(t_1)\|_{{H}^{1+c}} + \|v|v|^4\|_{{N}^{1+c}([t_1,t_2]\times \mathbb{R}^3)} + \|f(v,z)\|_{{N}^{1+c}([t_1,t_2]\times \mathbb{R}^3)} .
\end{equation}

So we need to bound the two remaining terms. 
\begin{lemma} \label{lemma61} If $v$ is a solution to $(\ref{eqn:61})$ then $\|v|v|^4\|_{{N}^{1+c}([t_1,t_2]\times \mathbb{R}^3)} \lesssim \epsilon^4\|v\|_{{S}^{1+c}([t_1,t_2]\times \mathbb{R}^3)} $.
\end{lemma}
Proof: Note that the pair $(\frac{10}{3},\frac{10}{3})$ is Schr\"odinger-admissible and has H\"older conjugate $(\frac{10}{7},\frac{10}{7})$. Therefore by equation (\ref{300}) we have
\begin{equation*}
\begin{split}
\|v\cdot |v|^4\|_{{N}^{1+c}([t_1,t_2]\times \mathbb{R}^3)} &\lesssim \|  \langle \nabla \rangle^{1+c}v \cdot |v|^4\|_{L^{10/7}L^{10/7}([t_1,t_2]\times \mathbb{R}^3)}\\
&\lesssim  \|  \langle \nabla \rangle^{1+c}v \|_{L^{10/3}L^{10/3}([t_1,t_2]\times \mathbb{R}^3)} \|v\|^4_{L^{10}L^{10}([t_1,t_2]\times \mathbb{R}^3)}\\
&\lesssim  \|v\|_{{S}^{1+c}([t_1,t_2]\times \mathbb{R}^3)} \|v\|^4_{L^{10}L^{10}([t_1,t_2]\times \mathbb{R}^3)}\\
&\lesssim  \epsilon^4\|v\|_{S^{1+c}([t_1,t_2]\times \mathbb{R}^3)}. \\
\end{split}
\end{equation*}

\begin{proposition} \label{61} Assume $0<c<\frac{1}{8}$, $f(v,z) = (v+z)|v+z|^4 - v|v|^4$ and that $z$ and $v$ satisfy the $R$ and $\epsilon$ bounds in the proposition, where $z$ is the linear solution and $v$ is the solution to (\ref{non}). Then we have
\begin{equation}
\begin{split}
\| f\|_{N^{1+c}([t_1,t_2]\times \mathbb{R}^3)} &\lesssim \|  \langle \nabla \rangle^{1+c} f\|_{L^{3/2}L^{18/13}([t_1,t_2]\times \mathbb{R}^3)}\\
 &\lesssim \sqrt{\epsilon^7R}( \sqrt{\epsilon R} + \|v_1(t_1)\|_{{H}^1(\mathbb{R}^3)}+\| \langle \nabla \rangle u^5\|_{L^{3/2}L^{18/13}([t_1,t_2]\times \mathbb{R}^3)}).\\
\end{split}
\end{equation}
\end{proposition}

Proof: Observing that $(3,\frac{18}{5})$ is Schr\"odinger-admissible, we have

\begin{equation}
\| f\|_{N^{1+c}([t_1,t_2]\times \mathbb{R}^3)} \le \|f\|_{L^{3/2}{W}^{1+c,18/13}([t_1,t_2]\times \mathbb{R}^3)} \lesssim \sup_{\|w_6\|_{L^{3}L^{18/5}([t_1,t_2]\times \mathbb{R}^3)} \le 1} \int_{t_1}^{t_2}\int_{x}  \langle \nabla \rangle^{1+c} [f] w_6 dw.
\end{equation}

The function $f(v,z)$ is a sum of terms of the form $w_1w_2w_3w_4z_5$ where each $w_i$ is either a $v$ or $z$ term. We  dyadicaly decompose these first five terms (not $w_6$), refer to $P_{N_i}w_i$ as $w_i$, and sum over all frequencies $N_1-N_5$, and combinations of $v,z$ in integrals of the form 
\begin{equation}
\| \langle \nabla \rangle^{1+c} f\|_{L^{3/2}L^{18/13}([t_1,t_2]\times \mathbb{R}^3)} \lesssim \sup_{\|w_6\|_{L^{3}L^{18/5}([t_1,t_2]\times \mathbb{R}^3)} \le 1}\int_{t_1}^{t_2}\int_{x}  \langle \nabla \rangle^{1+c} [w_1w_2w_3w_4v_5] w_6 dw.
\end{equation}
We can assume that the $1+c$ derivatives fall on the term with highest frequency. Before going through cases, we prove the following lemmas that combine interpolation with the bilinear estimate, Lemma \ref{visanbound}.

\begin{lemma} \label{lemma62} If $N_1 \le N_2$, then for any pair of dyadic components $v_1 = P_{N_1}v$, $z_2 = P_{n_2}z$ we have the bound: 
\begin{equation}
\begin{split}
\|v_1z_5\|_{L^{30/11}L^{15/8}([t_1,t_2] \times \mathbb{R}^3)} &\lesssim  N_2^{-1/4-s+\delta/2}\|v_1\|^{1/2}_{L^{10}L^{10}([t_1,t_2]\times \mathbb{R}^3)}\\
&\times(\|v_1(t_1)\|_{{H}^1(\mathbb{R}^3)}+\|  \langle \nabla \rangle u^5\|_{L^{3/2}L^{18/13}([t_1,t_2]\times \mathbb{R}^3)})^{1/2}\sqrt{\epsilon R} \\
\end{split}
\end{equation}
\end{lemma}
Proof: 

First note that for $N_1 \le N_2$ we have the bilinear estimate \ref{visanbound}:
\begin{equation}
\begin{split}
\|v_1z_2\|_{L^2L^2([t_1,t_2]\times \mathbb{R}^3)} &\lesssim N_2^{-1/2+\delta}N_1^{1-\delta}(\|v_1(t_1)\|_{L^2(\mathbb{R}^3)}+\|u^5\|_{L^{3/2}L^{18/13}([t_1,t_2]\times \mathbb{R}^3)})\|z(t_1)\|_{L^2(\mathbb{R}^3)}\\
\|v_1z_2\|_{L^2L^2([t_1,t_2]\times \mathbb{R}^3)} &\lesssim N_2^{-1/2-s+\delta}(\|v_1(t_1)\|_{{H}^1(\mathbb{R}^3)}+\|  \langle \nabla \rangle u^5\|_{L^{3/2}L^{18/13}([t_1,t_2]\times \mathbb{R}^3)})R.\\
\end{split}
\end{equation}
Also,  by H\"older's inequality
\begin{equation}
\begin{split}
\|v_1z_2\|_{L^{30/7}L^{30/17}([t_1,t_2]\times \mathbb{R}^3)} &\le \|v_1\|_{L^{10}L^{10}([t_1,t_2]\times \mathbb{R}^3)} \|z_2\|_{L^{30/4}L^{30/14}([t_1,t_2]\times \mathbb{R}^3)}\\
&\le N_2^{-s}\|v_1\|_{L^{10}L^{10}([t_1,t_2]\times \mathbb{R}^3)}\epsilon.\\
\end{split}
\end{equation}
Now note that 
\begin{equation}
\begin{split}
\frac{1}{30/11} &= \frac{1/2}{2} + \frac{1/2}{30/7}   \\
\frac{1}{15/8} &= \frac{1/2}{2} + \frac{1/2}{30/17} .\\
\end{split}
\end{equation}
Interpolating with exponents $\frac{1}{2}, \frac{1}{2}$ yields
\begin{equation}
\begin{split}
\|v_1z_5\|_{L^{30/11}L^{15/8}([t_1,t_2] \times \mathbb{R}^3)}  &\lesssim N_2^{-1/4-s+\delta/2}\|v_1\|^{1/2}_{L^{10}L^{10}([t_1,t_2]\times \mathbb{R}^3)}\\
&\times (\|v_1(t_1)\|_{{H}^1(\mathbb{R}^3)}+\|  \langle \nabla \rangle u^5\|_{L^{3/2}L^{18/13}([t_1,t_2]\times \mathbb{R}^3)})^{1/2}\sqrt{\epsilon R}.\\
\end{split}
\end{equation}

\begin{lemma} \label{lemma63}
If $N_2 \le N_1$ then for any pair of dyadic components $v_1 = P_{N_1}v$, $z_2 = P_{n_2}z$ we have the bound 
\begin{equation}
\|v_1z_2\|_{L^{30/11}L^{15/8}([t_1,t_2] \times \mathbb{R}^3)} \lesssim   N_1^{-3/4+\delta/2}N_2^{1/2-s}(\|v_1(t_1)\|_{{H}^1(\mathbb{R}^3)}+\|  \langle \nabla \rangle u^5\|_{L^{3/2}L^{18/13}([t_1,t_2]\times \mathbb{R}^3)})\sqrt{\epsilon R} .
\end{equation}
\end{lemma}
Proof: The bilinear estimate \ref{visanbound} tells us: 
\begin{equation}
\begin{split}
\|v_1z_2\|_{L^2L^2([t_1,t_2]\times \mathbb{R}^3)} &\lesssim N_1^{-1/2+\delta}N_2^{1-\delta}(\|v_1(t_1)\|_{L^2(\mathbb{R}^3)}+\|  \langle \nabla \rangle u^5\|_{L^{3/2}L^{18/13}([t_1,t_2]\times \mathbb{R}^3)})\|z(t_1)\|_{L^2(\mathbb{R}^3)}\\
\|v_1z_2\|_{L^2L^2([t_1,t_2]\times \mathbb{R}^3)} &\le N_1^{-3/2+\delta}N_2^{1-s}(\|v_1\|_{{H}^1(\mathbb{R}^3)}+\| \langle \nabla \rangle u^5\|_{L^{3/2}L^{18/13}([t_1,t_2]\times \mathbb{R}^3)})R . \\
\end{split}
\end{equation}

Also, by H\"older's inequality we have:
\begin{equation}
\begin{split}
\|v_1z_2\|_{L^{30/7}L^{30/17}([t_1,t_2]\times \mathbb{R}^3)} &\lesssim \|v_1\|_{L^{\infty}L^{2}([t_1,t_2]\times \mathbb{R}^3)} \|z_2\|_{L^{30/7}L^{15}([t_1,t_2]\times \mathbb{R}^3)}\\
&\lesssim N_1^{-1}N_2^{-s}\|v_1\|_{{S}^1([t_1,t_2]\times \mathbb{R}^3)}\epsilon .\\
\end{split}
\end{equation}

So with exponents $\frac{1}{2}, \frac{1}{2}$ we interpolate between the $L^2L^2$ and $L^{30/7}L^{30/17}$ bounds, and apply the Strichartz estimate to get:
\begin{equation}
\|v_1z_2\|_{L^{30/11}L^{15/8}([t_1,t_2] \times \mathbb{R}^3)} \lesssim N_1^{-3/4+\delta/2}N_2^{1/2-s}(\|v_1(t_1)\|_{{H}^1(\mathbb{R}^3)}+\| \langle \nabla \rangle u^5\|_{L^{3/2}L^{18/13}([t_1,t_2]\times \mathbb{R}^3)}) \sqrt{\epsilon R} .
\end{equation}

\begin{lemma} \label{lemma64}
If $N_1 \le N_2$ then for $z_1 = P_{N_1}z$, $z_2 = P_{n_2}z$ we have 
\begin{equation}
\|z_1z_2\|_{L^{30/11}L^{15/8}([t_1,t_2] \times \mathbb{R}^3)} \le N_2^{-1/4-s+\delta/2}N_1^{1/2-s} R\epsilon.
\end{equation}
\end{lemma}
Proof: The proof is identical to that of Lemma \ref{lemma62} except that $v_1$ has been replaced with $z_1$, which is put in a $L^{10}L^{10}$ norm.

In analyzing terms of the form $w_1w_2w_3w_4v_5$ there are two cases for where the highest 
\newline frequencies occur:
\begin{itemize}
\item{Case 1:} The highest frequency is on a $z$ term.

\item{Case 2:} The highest frequency is on a $v$ term. 
\end{itemize}

Throughout these cases we will utilize the facts that $\|v\|_{L^{q}L^{r}} \lesssim \|v\|_{{S}^{1}}$ for $\frac{2}{q}+\frac{3}{r} = \frac{1}{2}$ and $\|v\|_{L^{q}L^r} \lesssim \|v\|_{{S}^{0}}$ for $(q,r)$ Schr\"odinger admissible. We will also use the three above lemmas. 
Now we begin the analysis of cases. 

\begin{enumerate}

\item{Case 1:} In this case the highest frequency is on $z_5$.
We have all the derivatives falling on $z_5$. 
\begin{itemize}
\item{1.a $v_1w_2w_3w_4z_5$ case:}

Applying H\"older's inequality, Lemma 6.4, and our assumptions about $\epsilon$ we have:
\begin{equation}
\begin{split}
I &= \int_{t_1}^{t_2}\int_{x} v_1w_2w_3w_4   \langle \nabla \rangle^{1+c} z_5 w_6 dw\\
&\le N_5^{1+c}\|w_2\|_{L^{10}L^{10}([t_1,t_2]\times \mathbb{R}^3)}\|w_3\|_{L^{10}L^{10}([t_1,t_2]\times \mathbb{R}^3)}\\
&\times\|w_4\|_{L^{10}L^{10}([t_1,t_2]\times \mathbb{R}^3)} \|v_1z_5\|_{L^{30/11}L^{15/8}([t_1,t_2] \times \mathbb{R}^3)}   \|w_6\|_{L^{3}L^{18/5}}\\
&\le N_5^{3/4+c-s+\delta/2} \|w_2\|_{L^{10}L^{10}([t_1,t_2]\times \mathbb{R}^3)}\|w_3\|_{L^{10}L^{10}([t_1,t_2]\times \mathbb{R}^3)}\|w_4\|_{L^{10}L^{10}([t_1,t_2]\times \mathbb{R}^3)}\\
&\times \|v_1\|^{1/2}_{L^{10}L^{10}([t_1,t_2]\times \mathbb{R}^3)}(\|v_1(t_1)\|_{{H}^1(\mathbb{R}^3)}+\| \langle \nabla \rangle u^5\|_{L^{3/2}L^{18/13}([t_1,t_2]\times \mathbb{R}^3)})^{1/2}\sqrt{\epsilon R}\|w_6\|_{L^{3}L^{18/5}}.\\
\end{split}
\end{equation}

For $c < \frac{1}{8}$, $s > \frac{7}{8}$ and $\delta =0+$ the power of $N_5$ is negative, and the sum converges. The $w_i$ terms are all bounded by $\|v_i\|_{L^{10}L^{10}([t_1,t_2]\times \mathbb{R}^3)}$ or $\epsilon$. So this is bounded by $\epsilon^{4}R^{1/2}(\|v_1(t_1)\|_{{H}^1(\mathbb{R}^3)}+\| \langle \nabla \rangle u^5\|_{L^{3/2}L^{18/13}([t_1,t_2]\times \mathbb{R}^3)})^{1/2}$.

\item{1.b $z_1z_2z_3z_4z_5$ case:}

Applying H\"older's inequality, Lemma 6.6, and our $\epsilon$ bounds we have:
\begin{equation}
\begin{split}
I &= \int_{t_1}^{t_2}\int_{x} z_1z_2z_3z_4  \langle \nabla \rangle^{1+c} z_5 w_6 dw\\
&\le N_5^{1+c}\|z_2\|_{L^{10}L^{10}([t_1,t_2]\times \mathbb{R}^3)}\|z_3\|_{L^{10}L^{10}([t_1,t_2]\times \mathbb{R}^3)}\\
&\times\|z_4\|_{L^{10}L^{10}([t_1,t_2]\times \mathbb{R}^3)} \|z_1z_5\|_{L^{30/11}L^{15/8}([t_1,t_2] \times \mathbb{R}^3)}   \|w_6\|_{L^{3}L^{18/5}}\\
&\le  N_1^{1/2-s} N_5^{3/4-s+c + \delta}\|z_2\|_{L^{10}L^{10}}\|z_3\|_{L^{10}L^{10}}\|z_4\|_{L^{10}L^{10}} \epsilon R \|w_6\|_{L^{3}L^{18/5}}.\\
\end{split}
\end{equation}
For $s > \frac{7}{8}$ and $c < \frac{1}{8}$ and $\delta =0+$ both powers are negative and this is bounded by $\epsilon^{4}R$.

\end{itemize}
\item{Case 2:}

In this case the highest frequency falls on $v$, meaning $N_1 \ge N_2,\ldots, N_5$. We have, applying H\"older's inequality, Lemma 6.5, and our $\epsilon$ bounds, for $N_5 \le N_1$:

\begin{equation}
\begin{split}
I &= \int_{t_1}^{t_2}\int_{x}  \langle \nabla \rangle^{1+c}v_1w_2w_3w_4 z_5 w_6 dw\\
&\le N_1^{1+c}\|w_2\|_{L^{10}L^{10}([t_1,t_2]\times \mathbb{R}^3)}\|w_3\|_{L^{10}L^{10}([t_1,t_2]\times \mathbb{R}^3)}\\
&\times\|w_4\|_{L^{10}L^{10}([t_1,t_2]\times \mathbb{R}^3)} \|v_1z_5\|_{L^{30/11}L^{15/8}([t_1,t_2] \times \mathbb{R}^3)}   \|w_6\|_{L^{3}L^{18/5}}\\
&\le N_1^{-1/4+c + \delta/2} N_5^{1/2-s}\|w_2\|_{L^{10}L^{10}([t_1,t_2]\times \mathbb{R}^3)}\|w_3\|_{L^{10}L^{10}([t_1,t_2]\times \mathbb{R}^3)}\|w_4\|_{L^{10}L^{10}([t_1,t_2]\times \mathbb{R}^3)}\\
&\times  (\|v_1(t_1)\|_{{H}^1(\mathbb{R}^3)}+\| \langle \nabla \rangle u^5\|_{L^{3/2}L^{18/13}([t_1,t_2]\times \mathbb{R}^3)}) \sqrt{\epsilon R}\|w_6\|_{L^{3}L^{18/5}}.\\
\end{split}
\end{equation}

As in case 1a, the $\|w_i\|_{L^{10}L^{10}([t_1,t_2]\times \mathbb{R}^3)}$ terms are all bounded by 
$\epsilon$ or $\|v\|_{L^{10}L^{10}([t_1,t_2]\times \mathbb{R}^3)}$. Therefore for $c<\frac{1}{8}$ the sum over frequencies is bounded by \newline $\sqrt{\epsilon^7R}(\|v_1(t_1)\|_{{H}^1(\mathbb{R}^3)}+\| \langle \nabla \rangle u^5\|_{L^{3/2}L^{18/13}([t_1,t_2]\times \mathbb{R}^3)})$.

\end{enumerate}
So in all cases the integral is bounded by 
\begin{equation*}
\sqrt{\epsilon^7R}( \sqrt{\epsilon R} + \|v_1(t_1)\|_{{H}^1(\mathbb{R}^3)}+\| \langle \nabla \rangle u^5\|_{L^{3/2}L^{18/13}([t_1,t_2]\times \mathbb{R}^3)}).
\end{equation*}  This completes the proof of Proposition 6.3.

So combining Lemma \ref{lemma61}, Proposition \ref{61} and the fact that $\epsilon \ll 1$ we arrive at the following pair of inequalities: 

\begin{equation*}
\begin{split}
\|v\|_{{S}^{1+c}([t_1,t_2]\times \mathbb{R}^3)} &\lesssim \|v(t_1)\|_{{H}^{1+c}} + \|  \langle \nabla \rangle^{1+c}f\|_{L^{3/2}L^{18/13}([t_1,t_2]\times \mathbb{R}^3)} \\
\|  \langle \nabla \rangle^{1+c}f\|_{L^{3/2}L^{18/13}([t_1,t_2]\times \mathbb{R}^3)} &\lesssim  \sqrt{\epsilon^7R}( \sqrt{\epsilon R} + \|v_1(t_1)\|_{{H}^1(\mathbb{R}^3)}+\| \langle \nabla \rangle u^5\|_{L^{3/2}L^{18/13}([t_1,t_2]\times \mathbb{R}^3)}).\\
\end{split}
\end{equation*}

So all that remains is to bound 
\begin{equation} \label{79}
\| \langle \nabla \rangle u^5\|_{L^{3/2}L^{18/13}([t_1,t_2]\times \mathbb{R}^3)} \le \|  \langle \nabla \rangle v^5\|_{L^{3/2}L^{18/13}([t_1,t_2]\times \mathbb{R}^3)}+\|  \langle \nabla \rangle f\|_{L^{3/2}L^{18/13}([t_1,t_2]\times \mathbb{R}^3)}.
\end{equation}

First observe that 
\begin{equation*}
\begin{split}
\|  \langle \nabla \rangle v^5\|_{L^{3/2}L^{18/13}([t_1,t_2]\times \mathbb{R}^3)} &\le \| \langle \nabla \rangle v \cdot v^4\|_{L^{3/2}L^{18/13}([t_1,t_2]\times \mathbb{R}^3)}\\
&\le \| \langle \nabla \rangle v\|_{L^{15/4}L^{90/29}([t_1,t_2]\times \mathbb{R}^3)}\|v\|^4_{L^{10}L^{10}([t_1,t_2]\times \mathbb{R}^3)}\\
&\le \|v\|_{{S}^1([t_1,t_2]\times \mathbb{R}^3)} \|v\|^4_{L^{10}L^{10}([t_1,t_2]\times \mathbb{R}^3)}\\
&\lesssim \|  \langle \nabla \rangle v^5\|_{L^{3/2}L^{18/13}([t_1,t_2]\times \mathbb{R}^3)}\|v\|^4_{L^{10}L^{10}([t_1,t_2]\times \mathbb{R}^3)} \\
&+ (\|v(t_1)\|_{{H}^1(\mathbb{R}^3)}+\|f\|_{{N}^1([t_1,t_2]\times \mathbb{R}^3)})\|v\|^4_{L^{10}L^{10}([t_1,t_2]\times \mathbb{R}^3)}
\end{split}
\end{equation*}
and for $\|v\|_{L^{10}L^{10}([t_1,t_2]\times \mathbb{R}^3)}$ less than $\epsilon$ we have
\begin{equation} \label{80}
\|  \langle \nabla \rangle v^5\|_{L^{3/2}L^{18/13}([t_1,t_2]\times \mathbb{R}^3)}  \lesssim\epsilon^4(\|v(t_1)\|_{{H}^1(\mathbb{R}^3)}+ \|f\|_{{N}^1([t_1,t_2]\times \mathbb{R}^3)}).
\end{equation}

Noting that $\|v\|_{L^{10}L^{10}([t_1,t_2]\times \mathbb{R}^3)}$ is small and combining Proposition \ref{61}, (\ref{79}), and (\ref{80}) we have:
\begin{equation} \label{81}
\begin{split}
\| \langle \nabla \rangle ^{1+c} f\|_{L^{3/2}L^{18/13}([t_1,t_2]\times \mathbb{R}^3)} &\lesssim \sqrt{\epsilon^7R}( \sqrt{\epsilon R} + \|v(t_1)\|_{{H}^1(\mathbb{R}^3)}+\| \langle \nabla \rangle f\|_{L^{3/2}L^{18/13}([t_1,t_2]\times \mathbb{R}^3)})\\
&\lesssim  \sqrt{\epsilon^7R}( \sqrt{\epsilon R} + \|v(t_1)\|_{{H}^1(\mathbb{R}^3)}+ \| \langle \nabla \rangle^{1+c} f\|_{L^{3/2}L^{18/13}([t_1,t_2]\times \mathbb{R}^3)}).\\
\end{split}
\end{equation}

For $\epsilon  \ll \frac{1}{R}$ this implies that
\begin{equation*}
\| \langle \nabla \rangle^{1+c} f\|_{L^{3/2}L^{18/13}([t_1,t_2]\times \mathbb{R}^3)}  \lesssim \sqrt{\epsilon^7R}(\sqrt{\epsilon R} +  \|v(t_1)\|_{{H}^1(\mathbb{R}^3)}) .\\
\end{equation*}

This gives us the necessary bound on $f$.  

Combining this result with Lemma 6.2, we have
\begin{equation}
\|v\|_{{S}^{1+c}([t_1,t_2]\times \mathbb{R}^3)} \lesssim  \|v(t_1)\|_{{H}^{1+c}(\mathbb{R}^3)} + C(\epsilon)
\end{equation}
for sufficiently small $\epsilon \ll \frac{1}{R}$, which completes the proof of Proposition 6.1.

\bibliographystyle{plain}

Department of Mathematics, UC Berkeley, Berkeley CA 94720

\end{document}